\documentclass[11pt, reqno]{amsart}
\pdfoutput=1 
\usepackage[utf8x]{inputenc}
\usepackage[english]{babel}
\usepackage{graphicx}           
\usepackage{url}
\usepackage{eurosym}
\usepackage{ulem}
\usepackage{color} 
\usepackage[margin=3.3cm]{geometry}
\usepackage{setspace}
\usepackage{datetime}
\usepackage{pifont}
\usepackage{tikz}
\usepackage{natbib}
\usepackage{hyperref}
\usepackage{amsmath, amsthm,amsfonts,amssymb,latexsym,mathabx}

\newcommand{\C}{\mathcal{C}}			
\newcommand{\M}{\mathcal{M}}			
\newcommand{\X}{ {[0,M]} }

\newcommand{\DD}{\mathbb{D}}			
\newcommand{\EE}{\mathbb{E}}			
\newcommand{\PP}{\mathbb{P}}			
\newcommand{\RR}{\mathbb{R}}			

\newcommand{\taudiv}{b}							

\newcommand{\mmax}{M}    						
\newcommand{\rhog}{g}							
\newcommand{\Sin}{{\mathbf s}_{\textrm{\tiny\rm in}}} 
\newcommand{\gau}{G}

\renewcommand{\H}{\mathcal{H}}
\newcommand{\nub}{\widebar\nu}				
\newcommand{\intX}{\int_\X}

\newcommand{\dif}{\mathrm{d}}
\newcommand{\rmd}{{{\textrm{\upshape d}}}}

\newcommand{\norme}[1]{\left\Vert #1 \right\Vert }
\newcommand{\crochet}[1] {\langle #1 \rangle}
\newcommand{\trace}[1]{\crochet{\crochet{#1}}}

\newcommand{\ps}[2]{#1 \left(#2\right)}

\newtheorem{theorem}      {Theorem}[section]

\newtheorem{lemma}        [theorem]{Lemma}
\newtheorem{remark}       [theorem]{Remark}
\newtheorem{hypotheses}   [theorem]{Assumptions}

\usepackage{ifthen}    
\newboolean{showComments}        
\setboolean{showComments}{true}  

\begin{document}

\title[Gaussian approximations for chemostat models]{Gaussian approximations for chemostat models in finite and infinite dimensions}

\author[Bertrand Cloez \and Coralie Fritsch]{Bertrand Cloez$^{1}$ \and Coralie Fritsch$^{2,3,4}$}

\footnotetext[1]{INRA Montpellier 
UMR MISTEA , 
2 place Pierre Viala 
34060 Montpellier, France}


\footnotetext[2]{CMAP, \'Ecole Polytechnique, UMR CNRS 7641, route de Saclay, 91128 Palaiseau Cedex, France}

\footnotetext[3]{Universit\'e de Lorraine, Institut Elie Cartan de Lorraine,
    UMR CNRS 7502, 54506 Vand\oe uvre-l\`es-Nancy, France}

\footnotetext[4]{Inria, TOSCA, 54600 Villers-l\`es-Nancy, France\protect \\ 
				E-mail: {bertrand.cloez@supagro.inra.fr}, {coralie.fritsch@inria.fr}}

\begin{abstract}
In a chemostat, bacteria live in a growth container of constant volume in which liquid is injected continuously. Recently, Campillo and Fritsch introduced a mass-structured individual-based model to represent this dynamics and proved its convergence to a more classic partial differential equation. 

In this work, we are interested in the convergence of the fluctuation process. We consider this process in some Sobolev spaces and use central limit theorems on Hilbert space to prove its convergence in law to an infinite-dimensional Gaussian process.

As a consequence, we obtain a two-dimensional Gaussian approximation of the Crump-Young model for which the long time behavior is relatively misunderstood. For this approximation, we derive the invariant distribution and the convergence to it. We also present numerical simulations illustrating our results.\\


\paragraph{Mathematics Subject Classification (MSC2010):} 
60F05, 92D25, 60J25, 60G57, 60B10, 60H10
\end{abstract}

\maketitle


\section{Introduction}

The chemostat is a biotechnological process of continuous culture developed by \cite{monod1950a} and \cite{novick1950a} in which bacteria live in a growth container of constant volume in which liquid is continuously injected.

From a mathematical point of view, beyond classic models based on systems of ordinary differential equations  (see for instance \cite{smith1995a})  or integro-differential equations (see for instance \cite{fredrickson1967a,ramkrishna1979a,ramkrishna2000a}), several stochastic models were introduced in the literature. The first-one seems to be the one developed by  \cite{crump1979a} and is a birth and death process for the biomass growth coupled with a differential equation for the substrate evolution. This one is the main object of interest in Section \ref{sect:crump-young} below.  Recently,  \cite{campillo2011chemostat} and \cite{collet2013} studied some extensions of this model. In particular, \cite{campillo2011chemostat} propose some stochastic differential equations to model the demographic noise from the microscopic interactions.

Other stochastic models were introduced by \cite{stephanopoulos1979a,imhof2005a}. Let us also mention \cite{diekmann-odo2005a,mirrahimi2012a,mirrahimi2014a} or, for 
individual-based models, \cite{campillo2016a, campillo2016b, champagnat2014a, fritsch2016a} which model the evolutionary dynamics of the chemostat.

We focus here on the individual-based model  developed by \cite{campillo2014a} and \cite{fritsch2015a}. In this mass-structured model, the bacterial population is represented as a set of individuals growing in a perfectly mixed vessel of constant volume. This representation combines discrete mechanisms (birth and death events) as well as continuous mechanisms (mass growth and substrate dynamics).  \cite{campillo2014a} set the exact Monte Carlo simulation algorithm of this model and its mathematical representation as a stochastic process. They prove the convergence of this process to the solution of an integro-differential equation when the population size tends to infinity. 
In the present work, we investigate the study of the fluctuation process; namely the difference between the measure-valued stochastic process and its deterministic approximation. We first show that,
conveniently normalized, this fluctuation process converges to some superprocess. Our proof is based on a classic tightness-uniqueness argument in infinite dimension. In contrast with \cite{fournier2004a, campillo2014a, haskovec2011a}, one difficulty is that the main process is a signed measure and we have to find a suitable space in which
it, as well as its limit, are to be immersed (because the space of signed measures endowed with 
the weak convergence is not metrizable). Inspired by \cite{meleard1998} and \cite{vietchitran2006a}, we consider the fluctuation process as an element of some Sobolev space (see Section \ref{sect:spaces} for a description of this space). This type of spaces takes the advantage to be Hilbertian and one can use martingale techniques on Hilbert spaces to obtain the tightness (and then the convergence of this process); see for instance \cite{metivier1984a}. The limit object that we obtain is then an infinite dimensional degenerate Gaussian process. 

We illustrate the interest of this result applying it in finite dimension. More precisely, for particular parameters, the mass-structured model of \cite{campillo2014a} can be reduced to the two-dimensional Crump-Young model. As pointed out by \cite{collet2013}, the long time behavior of this model is complex and misunderstood; only few properties are known about the behavior before extinction. The convergence developed by \cite{campillo2014a} induces an approximation by an ordinary differential equation of the Crump-Young model, whereas our main result allows to obtain a stochastic differential approximation for which we are able to plainly describe the long-time behavior.

\medskip

Our main results are described in section that follows: Theorems \ref{th:tcl-intro} and \ref{th:tcl-CY} are the central limit theorems (convergence of the fluctuation processes) in infinite and finite dimensions. Theorem \ref{th:intro-CY} gives the long time behavior of a stochastic differential approximation of the Crump-Young model.
Section \ref{sect:CLT} is devoted to the proofs of the two central limit theorems. We first, introduce all the notations and preliminaries that we need from Section \ref{sect:spaces} to Section \ref{subsec.pre.res}, then Theorem \ref{th:tcl-intro} is proved in Section \ref{sec.proof.TCL}. The main steps of the proof of Theorem \ref{th:tcl-CY} are given in Section \ref{sec.proof.th:tcl-CY}.
The finite-dimensional case is studied in Section \ref{sect:crump-young}. We prove the convergence in time of the stochastic differential approximation of the Crump-Young model in Section \ref{sect:proofCY}. We present numerical simulations and discussion illustrating our results in Section \ref{subsec.num.sim}. In particular, we discuss about the validity of the approximation and introduce another diffusion process, obtained from Theorem \ref{th:tcl-CY}, whose numerical behavior seems to have a better mimic of the Crump-Young model in some particular situations.
The extinction time of this new process is studied in Section \ref{subsec.extin.newSDE}.

\subsection*{Main results}
Let us be more precise on our main results before to introduce all the machinery (notations, Sobolev spaces, ...) that we will use.

We consider the following mass-structured chemostat model : each individual is characterized by its mass $x\in \X$, where $\mmax$ is the maximal mass of a bacterium. At each time $t\geq 0$, the system is characterized by the random variable $(S^n_t,\, \nu^n_t)$, where $S^n_t$ is the substrate concentration and $\nu^n_t=\sum_{i=1}^{N_t^n}\delta_{X_t^i}$ is the population of the $N_t^n$ individuals with mass $X_t^1,\cdots,\,X_t^{N_t^n}$ in the chemostat at time $t$. The parameter $n$ represents a scaling parameter.

\medskip

We assume that one individual with mass $x\in \X$ 
\begin{itemize}
\item divides at rate $b(S,x)$ into two individuals with masses $\alpha\,x$ and $(1-\alpha)\,x$ where $\alpha$ is distributed according to a probability distribution $Q(\dif \alpha)$ on $[0,1]$;
\item is withdrawn from the chemostat at rate $D$, with $D$ the dilution rate of the chemostat;
\item grows at speed $g(S,x)$ : $\dot x_t = g(S^n_t,x_t)$,
\end{itemize}
where the substrate concentration $S^n$ evolves according to the following equation
$$
	\dot S^n_t = D\,(\Sin-S^n_t)-\frac{k}{n\,V}\,
	\sum_{i=1}^{N^n_t}g(S^n_t,X^i_t)\,,
	\qquad  S^n_0=S_0\,,
$$
where $\Sin$ is the input substrate concentration in the chemostat, $S_0$ is a deterministic initial substrate concentration, $n\,V$ is the volume of the chemostat and $k$ is a stoichiometric coefficient. Note that the scale parameter $n$ is only involved in front of the volume and the initial number of individuals. The approximations below then holds when the volume and the initial population become larger and larger. In this context, let us do a small remark on modelling. Parameter $D$ corresponds to a dilution rate, which is usually defined as the ratio between the flow and the volume. As we assume that the dilution rate is constant, approximations below only hold when also the flow became larger and larger.

 A more complete description of the stochastic process is given in Section \ref{sect:spaces} in term of martingale problem. To have a better understanding of the dynamics let also see \cite[Section 2.2]{campillo2014a}.

For every $n\geq 1$, we consider the renormalized process $(\widebar \nu^n_t)_{t\geq 0}$ defined by
\begin{align}
\label{def.nubar}
	\widebar \nu^n_t = \frac{1}{n}\,\nu^n_t, \quad \ t\geq 0
\end{align}
and we make the following assumptions.
\begin{hypotheses}[Regularity of the division rate and the growth speed]\ 
\label{hyp.model}
\begin{enumerate}
\item \label{hyp.lipch} The functions $ (s,x) \mapsto g(s,x)$ and $(s,x) \mapsto  \taudiv(s,x)$ are Lipschitz continuous w.r.t. $s$ uniformly in $x$ and differentiable in $s$ with derivative Lipschitz continuous w.r.t. $s$ uniformly in $x$: for all $s_1,s_2\geq 0, \ x\in [0,M]$,
$$ \begin{array}{cc}
              |g(s_1,x)-g(s_2,x)|\leq K_{gb} \, |s_1-s_2|; & |\taudiv(s_1,x)-\taudiv(s_2,x)|\leq K_{gb} \, |s_1-s_2|  ;       \\
              |\partial_s g(s_1,x)-\partial_s g(s_2,x)|\leq K_{gb} \, |s_1-s_2|;   &  |\partial_s \taudiv(s_1,x)-\partial_s \taudiv(s_2,x)|\leq K_{gb} \, |s_1-s_2 |.   
          \end{array}$$
\item \label{hyp.g} The function $g \in C^{1,1}(\RR_+\times \X)$ is such that $	g(s,0)=g(s,\mmax)=0\,.$
\item In absence of substrate the bacteria do not grow, i.e. $g(0,x)=0$ for all $x\in \X$.
\end{enumerate}
\end{hypotheses}

Note that due to the form of the differential equation satisfied by the substrate concentration $(S^n_t)_{t\geq 0}$, one can see that it remains in the compact set $[0, \max(S_0,\Sin)]$. As a consequence, the regularity of the functions $g$ and $b$, induced by Assumptions \ref{hyp.model}, implies that the division rate and the growth speed are bounded : 
$$0\leq \taudiv(s,x)\leq \widebar \taudiv, \qquad 0\leq g(s,x)\leq \widebar g, \qquad s\geq 0, \ x\in [0,M] .$$

With these assumptions, \cite{campillo2014a} show that if the sequence $(\widebar\nu_0^n)_n$ converges in distribution towards a deterministic, finite and positive measure $\xi_0$ 
then, under Assumptions \ref{hyp.model}, the following limit holds in distribution (see Section \ref{sect:spaces} for details about the topology),
\begin{equation}
\label{eq:cvEID}
\lim_{n \to \infty} (S_t^n,\nub_t^n)_{t\in [0,T]} = (S_t,\xi_t)_{t\in [0,T]}, 
\end{equation}
for any horizon time $T>0$, where $(S_t,\xi_t)_{t\in [0,T]}$ is the solution of the deterministic system of equations
\begin{equation}
\label{eq.limite.eid.faible}
     \begin{cases}
	S_t 	
	&=	S_{0}
	+
	\int_{0}^t\biggl[
	D\,(\Sin-S_u)-\frac kV \int_\X \rhog(S_u,x)\,\xi_u(\dif x)
	\biggr]\,\rmd u\,,
\\[1em]
 
	\ps{\xi_t}{f}
	&=
	\ps{\xi_0}{f}
	+
	\int_{0}^t \int_{\X} 
		\biggl[
			\taudiv(S_u, x)\,   \int_0^1
				\Bigl[f(\alpha\, x)+ f((1-\alpha)\,x)-f(x)\Bigr] \, 	 			
				Q(\dif \alpha)
\\
	&\qquad\qquad\qquad\qquad
	 - D\, f(x)+\rhog(S_u,x)\,f'(x)\biggr] \,	\xi_u(\dif x) 
	\,\rmd u\,,
     \end{cases}
\end{equation}
for any $f \in C^{1}(\X), t\geq 0$, with, for any $h\in \C(\X)$ and $\nu$ in the set $\M_F(\X)$ of finite (positive) measures,
\begin{align}
\label{def.ps}
	\ps{\nu}{h} := \int_\X h(x)\,\nu(\dif x)\,.
\end{align}

Let us finally introduce the main object of the present article, that is the fluctuation process $\widebar\eta^n_t = (\eta^n_t,R^n_t)$ defined by
\begin{align}
\label{eq:fluctuation}
\eta_t^n := \sqrt{n}\,\left(\nub_t^n-\xi_t \right), \qquad
R^n_t : = \sqrt{n} \, (S^n_t-S_t)\,.
\end{align}

Our main result is Theorem \ref{th:tcl-intro} below. 
For presentation convenience, we don't detail here the topology of $\DD([0,T],\H)$, $\H$ or $\H_0$ but all details are given in Section \ref{sect:spaces}, in particular $\H$ and $\H_0$ are defined in \eqref{def.H} (see also Remark \ref{remark.generalisation}). Briefly $\DD([0,T],\H)$ is the Skohorod space associated to an appropriately chosen Sobolev space $\H$ and $\H_0 \subsetneq \H$.

\begin{theorem}[Convergence of the fluctuation process]
\label{th:tcl-intro}
Under Assumption \ref{hyp.model} and if $\sup_{n \geq 1} \EE \left(\norme{\widebar \eta_0^n}^2_{\H_0} \right) < \infty$  and $(\widebar \eta^n_0)_{n \geq 1}$ converges to some $\widebar \eta_0 = (\eta_0,0)$ in $\mathcal{H}$ then, for any horizon time $T>0$, the sequence of process $(\widebar \eta^n)_{n \geq 1}$ converges in distribution in $\DD([0,T],\H)$ towards $\widebar \eta=(\eta,R)$ solution of the system

\begin{equation}
\label{eq.limite.tcl}
     \begin{cases}
	\ps{\eta_t}{f}
	&= 
	\ps{\eta_0}{f}
	+
	\int_0^t \intX \biggl[
	\taudiv(S_u,x) \, 
	 \int_0^1  
			  \bigl[ f(\alpha \, x)+f((1-\alpha)\,x)-f(x)\bigr] \, 
			  Q(\dif \alpha)
\\
	&\qquad	\qquad\qquad	\qquad 
		-D\, f(x) 
		+\rhog(S_u,x)\,f'(x) 
		\biggr]	  
		\,\eta_u(\dif x) \, \dif u
\\
	&\quad
	+
	\int_{0}^t R_u  \, \intX 
	\bigg[
	\partial_s\taudiv(S_u,x)\, 
	 \int_0^1
	 	\bigl[f(\alpha\, x)+ f((1-\alpha)\,x)-f(x)\bigr] \, 	 			
			Q(\dif \alpha) 
\\
	& \qquad \qquad \qquad \qquad
	+ 
   		\partial_s \rhog(S_u,x) \,f'(x) 
   		\biggr]
	 		\, \xi_u(\dif x) \,\dif u + G_t(f)
\\
R_t & = 
		-\int_{0}^t\biggl[
			D\,R_u+
			\frac kV \, \ps{\eta_u}{\rhog(S_u,\cdot)}
			+ R_u \, \frac kV \, \ps{\xi_u}{\partial_s\rhog(S_u,\cdot)} 
		\biggr]\,\rmd u
		\end{cases}
\end{equation}

where $\gau(f)$ is a centred Gaussian process with quadratic variation
\begin{align}
\nonumber
\crochet{\gau(f)}_t
	&=
		\int_0^t \intX \taudiv(S_u,x) \, 
	 		\int_0^1  
			  \bigl[ f(\alpha \, x)+f((1-\alpha)\,x)-f(x)\bigr]^2 \, 
			  Q(\dif \alpha)\,\xi_u(\dif x) \, \dif u
\\
\label{crochet.martin}
	&\quad
		+
		D\, \int_0^t \intX f^2(x) \, \xi_u(\dif x) \, \dif u,
\end{align}
for any $f\in C^1(\X)$ and $t\in [0,T]$.
\end{theorem}
Recall that the notations 
$$\ps{\eta_u}{\rhog(S_u,\cdot)}=\int_{[0,M]} \rhog(S_u,x)\,\eta_u(\dif x)\,,\qquad \ps{\xi_u}{\partial_s\rhog(S_u,\cdot)}=\int_{[0,M]}\partial_s\rhog(S_u,x)\,\xi_u(\dif x)$$ have been defined in \eqref{def.ps}.

\medskip

This theorem may look complicated but let us illustrate the interest of this type of result with a finite dimensional application. Let us choose
\begin{equation}
\label{eq:parameterCY}
M= \infty, \quad g(s,x)=\mu(s) m, \quad b(s,x) = \mu(s),
\end{equation}
where $s \mapsto \mu(s)$ is  the specific growth rate of the population that we will assume to be Lipschitz. Even though the previous assumptions are not included in the set of assumptions of Theorem \ref{th:tcl-intro} (see however remark \ref{rq:extension}), we can obtain the same result for specific functions $f$, in particular, when $f\equiv 1$ (see Theorem \ref{th:tcl-CY} below and its proof in Section \ref{sec.proof.th:tcl-CY}).

For parameters \eqref{eq:parameterCY}, the substrate concentration of the stochastic model satisfies
$$
	\frac{\dif}{\dif t} S^n_t = D\,(\Sin-S^n_t)-\frac{k}{V\,n}\,m\,\mu(S^n_t)\,N^n_t
$$
where the process $(N^n_t)_t$, depicting the number of individuals, is a birth-death process with non-homogeneous birth rate $\mu(S^n_t)$ and death rate $D$. It is exactly the Crump-Young model as studied in \cite{campillo2011chemostat,collet2013, crump1979a}. In particular, the long time behavior of this process is investigated in \cite{collet2013}. It is shown that this process extincts after a random time and, under suitable assumptions ($\mu$ increasing,...), admits (at least) a quasi-stationary distribution. This distribution describes the behavior of the process before the extinction (when it is unique and there is convergence to it); see for instance \cite{CMS13}. 

Let us define 
$$
\widebar N^n_t = \widebar{\nu}^n_t(1)=\frac{N^n_t}{n}, \quad N_t= \xi_t(1), 
\quad Q^n_t =\sqrt{n}(\widebar N^n_t-N_t).
$$
 We have then the following result.

\begin{theorem}[Convergence of the Crump-Young fluctuation process]
\label{th:tcl-CY}
If $\sup_{n\geq 1} \EE(|\widebar N_0^n|^2 +|Q_0^n|^2 )<+\infty$ and the sequence of random variables $(\widebar N_0^n,Q_0^n)_{n\geq 1}$ converges in distribution towards $(N_0,Q_0)$ then the sequence of processes $((\widebar N_t^n,S^n_t, Q_t^n, R_t^n)_{t\geq 0})_{n \geq 1}$ converges in distribution in $\DD([0,T],\RR^{4})$ towards $(N,S,Q,R)$ solution of the following system of stochastic differential equations:
\begin{equation}
\label{eq:CYEDS}
     \begin{cases}
\dif N_t &= (\mu(S_t)-D)\,N_t\,\dif t\,, \\
\dif S_t &= \left[D\,(\Sin-S_t)-\frac{k}{V}\,m\,\mu(S_t)\, N_t \right] \dif t\,,\\
\dif Q_t 	&= \left[ (\mu(S_t)-D) Q_t + \mu'(S_t) R_t N_t \right] \dif t	+ 
	\sqrt{(\mu(S_t)+D)\,N_t}\,\dif B_t\,, \\
\dif R_t &= -\left[	D\,R_t+	\frac kV \, \mu(S_t)\,m\,Q_t	+ \frac kV \,R_t \, \mu'(S_t)\,m\,  N_t \right] \dif t\,.
     \end{cases}
\end{equation}
for all $t\geq 0$, where $(B_t)_{t\geq 0}$ is a classic Brownian motion. 
\end{theorem}
The previous theorem suggests, if $n$ is sufficiently large, that
\begin{align*}
N_t^n \approx \widehat N_t^n := n\, N_t + \sqrt{n}\, Q_t, \quad 
S^n_t \approx \widehat S_t^n := S_t + \frac{1}{\sqrt{n}}\, R_t\,,
\end{align*}
with $(\widehat N_t^n, \widehat S_t^n)_{t\geq 0}$ solution of
\begin{align}
\label{SDE}
\begin{cases}
\dif \widehat N^n_t 
	&= 
		\big[(\mu(S_t)-D)\,\widehat N^n_t
			+ \mu'(S_t) (\widehat S^n_t-S_t)\,n\, N_t\big] \dif t
		+ \sqrt{(\mu(S_t)+D)\,n\,N_t}\,\dif B_t
\\
\dif \widehat S^n_t 
	&=
		\big[D\,(\Sin-\widehat S^n_t)-\frac{k}{V\,n}\,m\,\mu(S_t)\, \widehat N^n_t
			- \frac{k}{V\,n}\,\mu'(S_t) (\widehat S^n_t-S_t)\,n\, N_t \big] \dif t.
\end{cases}
\end{align}

Note that another (SDE type) approximation is given in Section \ref{subsec.num.sim}. This is a Feller-diffusion type approximation (see \cite{BM15}) and it is closer to the SDE introduced in \cite{campillo2011chemostat}.

The two first equations of \eqref{eq:CYEDS} are, up to a factor $m/V$ in front of $N_t$, the classic differential equations for representing the chemostat (see \cite{smith1995a}).
The four-component process is a non-elliptic diffusion time-homogeneous process whose long time behavior is given by Theorem \ref{th:intro-CY} below.

\begin{theorem}[Long time behavior of the Crump-Young SDE]
\label{th:intro-CY}
Assume that $\mu$ is strictly increasing on $[0, \Sin]$, $\mu(0)=0$ and $\mu(\Sin)>D$. There exists a unique $(N^*,S^*)$ such that 
$$
\mu(S^*)=D \,, \qquad
N^* = \frac{V }{km} (\Sin -S^*),
$$
and for any initial condition in $\mathbb{R}_+^*\times  \mathbb{R}_+ \times \mathbb{R} \times \mathbb{R}$, the process $((N_t,S_t,Q_t,R_t)^T)_{t\geq 0}$ ($T$ designs the transpose of the vector) converges in distribution to a Gaussian random variable with mean $\mathrm{m}$ and variance $\Sigma$ defined by
\begin{align}
\label{invariant.mesure}
\mathrm{m}=\begin{pmatrix}
   		N^*\\
   		S^* \\
   		0 \\
   		0	
	\end{pmatrix}, \qquad
\Sigma=	
	\begin{pmatrix}
   		0 & 0 & 0 & 0 \\
   		0 & 0 & 0 & 0 \\
   		0 & 0 & \alpha & -\frac{\mu(S^*)}{\mu'(S^*)} \\
   		0 & 0 & -\frac{\mu(S^*)}{\mu'(S^*)} & \beta \\
	\end{pmatrix},
\end{align}
where
\begin{align}
\label{def.alpha}
	\alpha=
	\frac{\left(\frac{k}{V}\,m\,\mu'(S^*)\,N^*+\frac 32 \, D\right)^2- \frac 54 \, D^2}
   			{\frac{k}{V}\,m\,\mu'(S^*)\,(D+\frac{k}{V}\,m\,\mu'(S^*)\,N^*)}
\end{align}
and
\begin{align}
\label{def.beta}
	\beta=\frac{k}{V}\,m\, \frac{D^2}{\mu'(S^*)\,(D+\frac{k}{V}\,m\,\mu'(S^*)\,N^*)} \,.
\end{align}
\end{theorem}

Some extensions of this Theorem are given in Section \ref{sect:crump-young} such as the rate of convergence and non-monotonic growth rate. 
The last Theorem gives the heuristic that, until extinction and if $n$ and $t$ are sufficiently large, the discrete model is almost distributed as a normal distribution:
\begin{align}
\label{invariant.mesure2}
(N^n_t, S^n_t) \approx \mathcal{N}\left( \begin{pmatrix}
   		 n\, N^*\\
   		 S^* 	
	\end{pmatrix} , 
		\begin{pmatrix}
   	  n\, \alpha & -\frac{\mu(S^*)}{\mu'(S^*)} \\
   	 -\frac{\mu(S^*)}{\mu'(S^*)} & \frac{\beta}{n} \\
	\end{pmatrix}
	\right).
\end{align} 

As one would expect, the number of bacteria is negatively correlated to the substrate rate: more individuals implies less food and vice versa. Recall that $(S_t^n + \frac{km}{V\,n} N^n_t)_{t\geq 0}$ is a martingale (\textit{i.e.} the total mass is in mean conserved in the container).

This theorem can be understood as a first step to fully describe the Crump-Young model such as in the case of the logistic model described in \cite{chazottes2015a}. Indeed, in Section \ref{subsec.num.sim}, we will see that, in large population, the quasi-stationary distribution of the Crump-Young model matches with the stationary distribution of its approximation. This is not trivial (and also not proved) because, for instance, the limits when $n\to \infty$ and $t\to \infty$ do not even commute! 
An example with a non-monotonic $\mu$ with different behaviors (several invariant measures, behavior depending on the initial conditions) is also presented.

\section{Central limit theorems}
\label{sect:CLT}

\subsection{Functional notations}
\label{sect:spaces}

For any $n\geq 1$ and $t\geq 0$, the population of bacteria is represented by the punctual measure $\nu^n_t= \sum_{i=1}^{N^n_t} \delta_{X^i_t}$. We denote by $\mathcal{M}(\X)$ the set of such measures (punctual measures), it is a subset of the set $\M_F(\X)$ of finite (positive) measures.

For any $n\geq 1$ and $T>0$, the process $(\nu^n_t)_{t\in [0,T]}$ is a càd-làg process. It (almost-surely) belongs to the space $\DD([0,T],\M_{F}(\X))$ of càd-làg functions of $[0,T]$ with values in $\M_{F}(\X)$, endowed with the (usual) Skohorod topology; see for instance \cite{billingsley1968a,ethier1986a} for an introduction. In contrast, $(S^n_t)_{t\in [0,T]}$ is (almost surely) a continuous function. We denote by $\C([0,T],\RR_{+})$ the set of continuous functions from $[0,T]$ to $\RR_+$ endowed with its usual topology.

The convergence \eqref{eq:cvEID} corresponds to a convergence in distribution in the product space $\C([0,T],\RR_{+}) \times \DD([0,T],\M_{F}(\X))$. Roughly this convergence is proved by a compactness/uniqueness argument. The compactness (or tightness) is proved by the (well-known) Aldous criterion which is a stochastic generalisation of the Arzel\`{a}-Ascoli Theorem. One of the key assumption of this theorem is to work in metric space. Considering the fluctuation process $(\eta^n_t)_{t \geq 0}$, such arguments cannot be used to establish any convergence. Indeed, in contrast to the measure $\bar \nu^n_t$, the measure $\eta^n_t$ is not a positive measure but it is a signed measure. The set of signed measures being not metrisable
\citep{Varadarajan1958}, one has to consider $\eta^n_t$ as an operator acting on a different space than those of continuous and bounded functions.
As \cite{meleard1998} and \cite{vietchitran2006a}, we will use some Sobolev spaces that are defined as follows: for every integer $j$, we let $\C^j(\X)$ be the set of functions being $j$ times continuously differentiable endowed with the norm $\Vert \cdot \Vert_{\C^j}$, defined for all $f \in \C^j(\X)$ by $\Vert f \Vert_{\C^j} = \sum_{i=0}^j \Vert f^{(i)} \Vert_\infty$ (with $\Vert \cdot \Vert_\infty$ the infinity norm). Let now $\norme{\cdot}_{W_j} $ be the norm defined, for any $f \in \C^j(\X)$, by 
\begin{align*}
	\norme{f}^2_{W_j} := \intX \sum_{i=0}^j (f^{(i)}(x))^2\,\dif x . 
\end{align*}
Let $W_j=W_j(\X)$ be the completion of $\C^j(\X)$ with respect to this norm (note that it is the classical Sobolev space associated to the classical norm $\norme{\cdot}_{L^2}$). Contrary to the Banach space $\C^j(\X)$, the Sobolev space $W_j$ is Hilbertian. 
Let $W_j^*$ be its dual space, classically endowed with the norm
\begin{align*}
\norme{\mu}_{W_j^*}=\sup_{\norme{f}_{W_j}\leq 1} |\ps{\mu}{f}|\,.
\end{align*}
Another useful property is given by the Sobolev-type inequalities: there exist universal constants $C_j,C'_j$ such that
\begin{align}
\label{ineg.normes}
\norme{f}_{W_j} \leq C_j\, \norme{f}_{\C_j}, \quad \norme{f}_{\C_j}  \leq C'_j\,\norme{f}_{W_{j+1}}.
\end{align}

See for instance \cite[Equations (3.5) and (3.6)]{meleard1998} or \cite[Theorem V-4]{Adams}. In particular, $W_{j+1}$ is continuously embedded in $W_{j}$. Moreover, this embedding is a Hilbert-Schmidt embedding (see \cite[Equation (3.7)]{meleard1998} 
or \cite[Theorem VI-53]{Adams}). Therefore bounded and closed sets of $W_{j+1}$ are compact for the $W_j$'s topology.

Let us illustrate an application of inequalities \eqref{ineg.normes} that will be useful in the proof of Theorem \ref{th:tcl-intro}. 

\begin{lemma}[Useful bound on the basis] Let $(e_k)_{k\geq 0}$ be an orthonormal basis of $W_2$, we have,
\label{lemma.somme.ek}
$$
K_1 := \sup_{x\in[0,\mmax]} \sum_{k \geq 0} e_k^2(x)<+ \infty\,.
$$
\end{lemma}

\begin{proof}
By definition, for any $x\in[0,\mmax]$,
$$
	\norme{\delta_x}_{W^*_2} 
	= \sup_{\norme{f}_{W_2}\leq1}|\ps{\delta_x}{f}|
	= \sup_{\norme{f}_{W_2}\leq1}|f(x)|
	\leq \sup_{\norme{f}_{W_2}\leq1}\norme{f}_\infty\,.
$$
Moreover, we have the Sobolev-type inequalities, $\norme{f}_\infty= \norme{f}_{\C^0}\leq C\,\norme{f}_{W_2}$, for some $C>0$. Hence, by the Parseval identity,
\begin{align*}
C^2 & \geq
	\sup_{x\in[0,\mmax]}\norme{\delta_x}^2_{W^*_2} 
	= \sup_{x\in[0,\mmax]} \sum_{k\geq 0} e_k^2(x)\,.
\qedhere
\end{align*}
\end{proof}

Finally, contrary to the models of \cite{meleard1998,vietchitran2006a}, the fluctuation process (as the empirical measure) is not here a Markov process by itself. We have to consider the couple population/substrate to have a homogeneous dynamics. As a consequence, we will use a slightly larger space than those of \cite{meleard1998,vietchitran2006a}. Let
\begin{align}
\label{def.H}
\H_0 = W_2^*\times \RR, \quad \H = W_3^*\times \RR,
\end{align}
be the Hilbert spaces endowed with the following norms : for  $(\mu,R)\in\H_0$ or $(\mu,R)\in\H$,
$$\norme{(\mu,R)}_{\H_0} = \sqrt{\norme{\mu}_{W_2^*}^2+|R|^2}\,,\quad \norme{(\mu,R)}_\H = \sqrt{\norme{\mu}_{W_3^*}^2+|R|^2}\,.$$

\begin{remark}[Weaker assumptions on $(\bar \eta_0^n)_{n\geq 1}$]
\label{remark.generalisation}
More generally, Theorem \ref{th:tcl-intro} holds for $\H_0 = W_j^*\times \RR$ and $\H = W_{j+1}^*\times \RR$ with $j\geq 2$ (the entire proof holds replacing $W_2$, $W_2^*$, $W_3$ and $W_3^*$ by $W_j$, $W_j^*$, $W_{j+1}$ and $W_{j+1}^*$). For $j>2$, the assumptions on $(\bar \eta_0^n)_{n\geq 1}$ are weaker than for $\H_0$ and $\H$ defined by \eqref{def.H}, however the convergence result is also weaker.
\end{remark}

\subsection{Martingale properties}
\label{sec.martingale.properties}

The sequence of processes $((\widebar \nu_t^n)_{t\geq 0})_{n\geq 1}$, defined by \eqref{def.nubar}, can be rigorously defined as a solution of stochastic differential equations involving Poisson point processes; see \cite[Section 4]{campillo2014a}. Instead of using this characterisation, we only need that it is solution to the following martingale problem.
\begin{lemma}[Semi-martingale decomposition, \cite{campillo2014a}]
\label{lem:martinagle}
We assume that  $\EE(\ps{\nub_0^n}{1}^2)<\infty$. Let $f \in \C^1(\X)$,
then, under Assumptions \ref{hyp.model}, for all  $t>0$:
\begin{align}
\nonumber
  \ps{\nub_t^n}{f}
  &=
  \ps{\nub_0^n}{f}
  + \int_0^t \int_{\X} \biggl[
  	\taudiv(S^n_u,x) \, 
	 \int_0^1  
			  \bigl[ f(\alpha \, x)+f((1-\alpha)\,x)-f(x)\bigr] \, 
			  Q(\dif \alpha)		  
\\		
\label{renormalisation}
  &\qquad  \qquad\qquad\qquad\quad
  - D\,f(x) +\rhog(S_u^n,x)\,f'(x) 
	 \biggr]\nub_u^n(\dif x) \, \dif u	  
  + \ps{Z_t^n}{f} 
\end{align}
where
\begin{align*}
  \frac{\dif}{\dif t} S_t^n 
  &=
  D(\Sin-S_t^n)-\frac {k}{V} \, 
    \int_{\X} \rhog(S_t^n,x)\,\nub_t^n(\dif x)
\end{align*}
and $(\ps{Z_t^n}{f})_{t\geq 0}$ is a martingale with the following predictable quadratic variation:
\begin{align*} 
\nonumber
  \crochet {\ps{Z^n}{f}}_t
  &=
  \frac 1n  
  \int_0^t \int_{\X} \taudiv(S^n_u, x) \, 
     \int_0^1 
        \left[ f(\alpha \, x)+f((1-\alpha) \, x)-f(x) \right]^2 \, 
			  Q(\dif \alpha) \,\nub_u^n(\dif x) \, \dif u  
\\
	&\qquad\qquad\qquad
	 + \frac 1n \, D\, \int_0^t \int_{\X} 
			f(x)^2 \, \nub_u^n(\dif x) \, \dif u\,. 
\end{align*}
\end{lemma}

Let us recall that $\widebar\eta^n_t = (\eta^n_t,R^n_t)$ is defined by \eqref{eq:fluctuation}. From the last lemma, we deduce that $\widebar\eta^n_t = \widebar A_t^n + \widebar M_t^n$ with
$$
	 \ps{\widebar A_t^n}{f} := 
		\left(
			\begin{tabular}{c}
				$\ps{A_t^n}{f}$\\
				$R_t^n$	
			\end{tabular}	 
		\right)
\qquad 	
	 \ps{\widebar M_t^n}{f} := 
		\left(
			\begin{tabular}{c}
				$\ps{M_t^n}{f}$\\
				0	
			\end{tabular} 
		\right)\,,
$$
where the processes $(A_t^n)_{t\in[0,T]}$ and $(R_t^n)_{t\in[0,T]}$ have finite variations and are defined by
\begin{align}
\nonumber
\ps{A_t^n}{f}
	&:= 
	\ps{\eta^n_0}{f}
	+
	\int_0^t \biggl[ 
	\intX \taudiv(S^n_u,x) \, 
	 \int_0^1  
			  \bigl[ f(\alpha \, x)+f((1-\alpha)\,x)-f(x)\bigr] \, 
			  Q(\dif \alpha)\,\eta_u^n(\dif x)
\\
\nonumber 
  &\quad
	 	-D\, \ps{\eta_u^n}{f}
	 	+\ps{\eta_u^n}{\rhog(S_u^n,.)\,f'} 
    \biggr]\, \dif u
	+ \sqrt{n} \,
   \int_0^t \biggl\{
	 		\ps{\xi_u}{\left(\rhog(S_u^n,.)-\rhog(S_u,.)\right)\,f'} 
\\
\label{def.Ant}
	&\quad
	+
	 \intX \left(\taudiv(S^n_u,x)-\taudiv(S_u, x)\right)\,  
	 \int_0^1
	 	\Bigl[f(\alpha\, x)+ f((1-\alpha)\,x)-f(x)\Bigr] \, 	 			
			Q(\dif \alpha) \,\xi_u(\dif x) \biggr\} \,\dif u 
\end{align}
and
\begin{align}
\label{def.Rnt}
R_t^n & =
	\int_{0}^t\biggl[
			- D \, R^n_u
			- \frac kV \, \ps{\eta^n_u}{\rhog(S^n_u,.)}
			- \sqrt{n} \, \frac kV \, \ps{\xi_u}{g(S^n_u,.)-g(S_u,.)}
	\biggr]\,\rmd u.
\end{align}
The process $(\ps{M_t^n}{f})_{[0,T]}$, defined by
\begin{align}
\label{def.Mn}
\ps{M_t^n}{f} := \sqrt{n}\, \ps{Z_t^n}{f}\,,
\end{align}
is a martingale with predictable quadratic variation
\begin{align} 
\nonumber
  \crochet{\ps{M^n}{f}}_t
  &= 
  \int_0^t \intX \taudiv(S^n_u, x) \, 
     \int_0^1 
        \left[ f(\alpha \, x)+f((1-\alpha) \, x)-f(x) \right]^2 \, 
			  Q(\dif \alpha) \,\nub_u^n(\dif x) \, \dif u  
\\
\label{var_qua}
	&\quad
	 +D\, \int_0^t \intX
			f(x)^2 \, \nub_u^n(\dif x) \, \dif u\,. 
\end{align}

\subsection{Preliminary estimates}
\label{subsec.pre.res}

For any $n\geq 1$, let $(\trace{M^n}_t)_{t\in [0,T]}$ be the trace of the process $(M^n_t)_{t \in [0,T]}$, that is the process such that $(\norme{M^n_t}_{W^*_2}^2-\trace{M^n}_t)_{t\in[0,T]}$ is a local martingale \citep{joffe1986a, metivier1982a, metivier1984a}. Let $(e_k)_{k\geq 0}$ be an orthonormal basis of $W_2$, we have
\begin{equation}
\label{eq:trace}
\trace{M^n}_t = \sum_{k\geq 0} \crochet{M^n(e_k)}_t\,.
\end{equation}
Indeed, on the one hand, the Parseval's identity entails that
\begin{align*}
	\norme{M^n_t}^2_{W_2^*} = \sum_{k\geq 0} (M_t^{n}(e_k))^2\, .
\end{align*}
On the other hand, by definition of the predictable quadratic variation, $(M^n_t(e_k)^2-\crochet{M^n(e_k)}_t)_{t \geq 0}$ is a martingale for any $k\geq 0$. Therefore, the uniqueness of the trace implies \eqref{eq:trace}.

\begin{lemma}[Uniform moment]
\label{esp.fini}
If $\sup_{n \geq 1} \EE \left(\norme{\widebar \eta_0^n}^2_{\H_0} \right) < \infty$ then for any $T>0$
$$
	\sup_{n \geq 1} \EE \left(\sup_{u \leq T} \norme{\widebar \eta_u^n}^2_{\H_0} \right) < \infty \,.
$$
\end{lemma}

\begin{proof}
By definition of $\norme{.}_{\H_0}$ and the triangular inequality, we have
\begin{align*}
\norme{\widebar\eta_t^n}^2_{\H_0}
	&= 
		\norme{\eta_t^n}^2_{W_2^*} + |R_t^n|^2
 \leq
		2\,(\norme{A_t^n}^2_{W_2^*} + \norme{M_t^n}^2_{W_2^*}) + |R_t^n|^2\,.
\end{align*}
By the Fubini-Tonelli theorem, 
\begin{align*}
\EE \left[ \sup_{u\leq t} \norme{M_u^n}^2_{W_2^*}\right]
	&=
		\EE \left[ \sup_{u\leq t} \sum_{k\geq 0} M_u^n(e_k)^2\right]
	\leq
		\sum_{k\geq 0}\EE \left[ \sup_{u\leq t} M_u^n(e_k)^2\right]\,.
\end{align*}
Applying the Doob inequality (see for instance \cite[Theorem 1.43]{jacod2003a}) to the martingale $(M^n_u(e_k))_{u\geq 0}$ for any $k \geq 0$, we have
\begin{align*}
\EE \left[ \sup_{u\leq t} \norme{M_u^n}^2_{W_2^*}\right]
	&\leq 
		4\, \sum_{k\geq 0}\EE \left[M_t^n(e_k)^2\right]
	=
		4\, \sum_{k\geq 0}\EE \left[\crochet{M^n(e_k)}_t\right]
	=
		4\, \EE \left[\sum_{k\geq 0} \crochet{M^n(e_k)}_t\right]\,.
\end{align*}
By \eqref{var_qua} and by Lemma \ref{lemma.somme.ek},
\begin{align*}
\sum_{k\geq 0} \crochet{M^n(e_k)}_t
	&=
		\sum_{k \geq 0}
		\int_0^t \intX \Big\{
			\taudiv(S_u^n,x)\, \int_0^1 
				\left[e_k(\alpha\,x)+e_k((1-\alpha)\,x)-e_k(x)\right]^2
				\,Q(\dif \alpha)  
\\
	& \qquad \qquad \qquad \quad
			+  D\,e^2_k(x)
				\Big\} \, \widebar\nu_u^n(\dif x)\,\dif u
\\
	& \leq
		\left(5\,\widebar \taudiv+D \right)\,K_1 \, T\, \sup_{u\leq t} \widebar \nu_u^n(1)\,.
\end{align*}
By assumption $\sup_n \EE \left(\norme{\widebar \eta_0^n}_{\H_0} \right) < \infty$, which implies that  $\sup_n \EE\left(\widebar\nu_0^n(1) \right)<+\infty$. Then, from \cite[Lemma 5.4]{campillo2014a}
$\sup_n \EE\left[\sup_{u\leq T} \widebar\nu_u^n(1) \right]<+\infty$. Hence
\begin{align}
\label{eq:borne-martinagle}
K_2:= \sup_n \EE \left[ \sup_{u\leq T} \norme{M_u^n}^2_{W_2^*}\right]
	<+\infty\,.
\end{align}
Using the Gronwall inequality, we easily check that 
\begin{align}
\label{sup.xi}
	 C_T^\xi := \sup_{t\in[0,T]} \ps{\xi_t}{1} < + \infty \,;
\end{align}
it is however proved in \cite[Proof of Theorem 5.2]{campillo2014a}.
Therefore, from Assumptions \ref{hyp.model} and the Sobolev-type inequalities,
\begin{align*}
\norme{A_t^n}_{W_2^*}
	& \leq
		\norme{\eta^n_0}_{W_2^*}
		+
		C\,\int_0^t \norme{\eta_u^n}_{W_2^*}\,\left(\widebar \rhog + 3\widebar \taudiv + D\right) \, \dif u
		+
		4\,C\, \int_0^t |R^n_u|\,K_{gb}\,C_T^\xi\,\dif u
\\
	& \leq
		\norme{\eta^n_0}_{W_2^*}
		+
		K_3\,\int_0^t \left(\norme{\eta_u^n}_{W_2^*} + |R^n_u| \right) \, \dif u \, ,
\end{align*}
with $K_3:=C\,\max\big\{\widebar \rhog + 3\widebar \taudiv + D ; 4\,K_{gb}\,C_T^\xi\big\}$ and $C$ depends on the constants $C_1$ and $C'_1$ of the Sobolev-type inequalities.
Hence
\begin{align*}
\norme{A^n_t}^2_{W_2^*}
	 \leq 2\,\left(
	 			\norme{\eta^n_0}_{W_2^*}^2
	 			+K_3^2\,T\,\int_0^t \norme{\widebar \eta_u^n}_{\H_0}^2\,\dif u
	 		\right)\,. 		
\end{align*}
Moreover,
\begin{align*}
\left|R_t^n \right|
	& \leq
		\int_0^t\left(
			D\, |R_u^n|
			+ C\,\frac{k}{V}\,\widebar \rhog \, \norme{\eta_u^n}_{W_2^*}	
			+\frac{k}{V}\, |R_u^n|\, K_{gb}\,C_T^\xi	
		 \right)\,\dif u\,
\end{align*}
hence
\begin{align*}
\left|R_t^n \right|^2
	& \leq
		 K_4^2\, T\,\int_0^t \norme{\widebar \eta_u^n}^2_{\H_0}	\,\dif u\, ,
\end{align*}
with $K_4=\max\left\{C\,\frac{k}{V}\,\widebar \rhog\,;\, D+ \frac{k}{V}\, K_{gb}\,C_T^\xi\right\}$.
Therefore, by the Fubini-Tonelli theorem
\begin{align*}
\EE \left[\sup_{s\leq t}\norme{\widebar\eta_s^n}^2_{\H_0}\right]
	& \leq
		(4\,K^2_3+K^2_4)\,T\, \int_0^t \EE \left[\sup_{s\leq u}\norme{\widebar\eta_s^n}^2_{\H_0}\right]\,\dif u
		 + 4\,\EE \norme{\widebar \eta_0^n}^2_{\H_0} + 2\,K_2\,.
\end{align*}
By Gronwall lemma, we finally get
\begin{align*}
	\sup_{n\geq 1}\EE \left[\sup_{s\leq t}\norme{\widebar\eta_s^n}^2_{\H_0}\right]
		 \leq
			e^{(4\,K^2_3+K^2_4)\,T\,t}\,(2\,K_2 + 4\,\sup_{n\geq 1}\EE \norme{\widebar\eta_0^n}^2_{\H_0})
		< + \infty\,.
\end{align*}
\end{proof}

\subsection{Proof of the Theorem \ref{th:tcl-intro}}
\label{sec.proof.TCL}

We divide the proof in two steps. The first one is devoted to the proof of the tightness of the sequence of processes $(\widebar\eta^n)_{n\geq 1}$ in $\DD([0,T],\H)$. In the second one, we prove that the limit of the process is unique and given by (\ref{eq.limite.tcl}-\ref{crochet.martin}).

\subsubsection*{Step 1 : Tightness of $(\widebar\eta^n)$ in $\DD([0,T],\H)$}
From \cite[Lemma C]{meleard1998} or \cite{joffe1986a}, the sequence of processes $((\widebar \eta^n)_{t\in{[0,T]}})_{n\geq 1}$ is tight in $\DD([0,T],\H)$ if the two following conditions hold:
\begin{itemize}
\item[${[T]}$] for all $t\leq T$,
$$
\sup_{n\geq 1} \mathbb{E}\left[\Vert \widebar \eta^n_t \Vert^2_{\mathcal{H}_0}  \right] <+ \infty;
$$
\item[${[A]}$] for any $\varepsilon>0$, $\alpha>0$, there exist $\theta>0$ and $n_0$ such that for any sequence $(\sigma_n,\tau_n)_n$ of pairs of stopping times with $\sigma_n \leq \tau_n \leq \sigma_n+\theta$,
\begin{align*}
	\sup_{n\geq n_0} 
		\PP(
			\norme{\widebar A^n_{\tau_n}-\widebar A^n_{\sigma_n}}_\H\geq \alpha
			)\leq \varepsilon \, ,
\end{align*}
\begin{align*}
	\sup_{n\geq n_0} 
		\PP(
			\left|\trace{M^n}_{\tau_n}-\trace{M^n}_{\sigma_n}\right|\geq \alpha
			)\leq \varepsilon \, .
\end{align*}
\end{itemize}
Indeed recall that the embedding $\H_0 \subset \H$ is Hilbert-Schmidt and then, using Markov inequality, $[T]$ implies that the sequence $(\widebar \eta^n_t)_n$ almost surely belongs to a bounded set of $\H_0$ (which is compact in $\H$).
In short, $[T]$ implies the tightness of $(\widebar \eta^n_t)_{n\geq 0}$ for every $t\geq 0$ in $\H$. 

In order to prove the tightness of $(\widebar\eta^n)$ in $\DD([0,T],\H)$, we have to prove the conditions $[T]$ and ${[A]}$. Condition $[T]$ is a direct consequence of Lemma \ref{esp.fini}. Let us now prove ${[A]}$. By the Markov inequality,
$$
\PP\left(\norme{\widebar A_{\tau_n}^n-\widebar A_{\sigma_n}^n}_\H \geq \alpha \right)
	\leq
	\frac{\EE\norme{\widebar A_{\tau_n}^n-\widebar A_{\sigma_n}^n}_\H}{\alpha}\, .
$$
By \eqref{def.Ant}, we have for any $f\in W_3 \subset W_2$ such that $\norme{f}_{W_3}\leq 1$,
\begin{align*}
|\ps{(A_{\tau_n}^n-A_{\sigma_n}^n)}{f}|
	& \leq
		C\,\int_{\sigma_n}^{\sigma_n+\theta} \norme{\eta_u^n}_{W_2^*}\,(\widebar \rhog + 3\,\widebar \taudiv+D)\,\dif u
		+ C\,\int_{\sigma_n}^{\sigma_n+\theta} |R^n_u|\,4\,K_{gb}\,C_T^\xi\,\dif u
\\
	& \leq
		C\,\int_{\sigma_n}^{\sigma_n+\theta} (\norme{\eta_u^n}_{W_2^*}+|R^n_u|)\,\dif u
\\
	& \leq
		C\, \theta \sup_{u\leq T} \norme{\widebar \eta_u^n}_{\H_0}
\end{align*}
where the constant $C$ can be different from a line to another.

By the same way,
$|R_{\tau_n}^n-R_{\sigma_n}^n|
	\leq
		C\, \theta \sup_{u\leq T} \norme{\widebar \eta_u^n}_{\H_0}$
then
$$
	\EE(\norme{\widebar A_{\tau_n}^n-\widebar A_{\sigma_n}^n}_\H)
		\leq 
		C\,\theta \, \sup_{n\geq 1} \EE \left(\sup_{u \leq T} \norme{\widebar \eta_u^n}_{\H_0} \right)\,.
$$
By Lemma \ref{esp.fini}, the first condition of ${[A]}$ is then satisfied.

\medskip

In the same way,
\begin{align*}
\left|\trace{M^n}_{\tau_n}-\trace{M^n}_{\sigma_n}\right|
	&= 
		\Bigg|
			\sum_{k\geq 0} 
				\big(
					\crochet{M^n(e_k)}_{\tau_n}
					-
					\crochet{M^n(e_k)}_{\sigma_n}
				\big)
		\Bigg|
\\
	&\leq
		\sum_{k\geq 0} 
		\left|
		\crochet{M^n(e_k)}_{\tau_n}
					-
					\crochet{M^n(e_k)}_{\sigma_n}
		\right|
\end{align*}
By \eqref{var_qua} and Lemma \ref{lemma.somme.ek}, we then get
$$
	\left|\trace{M^n}_{\tau_n}-\trace{M^n}_{\sigma_n}\right|
	\leq
		C\, K_1\, \int_{\sigma_n}^{\sigma_n+\theta} \ps{\nub_u^n}{1}\,\dif u\,
$$
therefore
$$
	\EE\left|\trace{M^n}_{\tau_n}-\trace{M^n}_{\sigma_n}\right|
	\leq
		C\, K_1\, \theta \, 
		\sup_{n\geq 1} \EE \left(\sup_{u \leq T} \norme{\widebar \eta_u^n}_{\H_0} \right)
$$
and by Lemma \ref{esp.fini}, the condition ${[A]}$ holds.

\subsubsection*{Step 2 : Identification of the accumulation points}

From Step 1, the sequence $(\widebar\eta^n)_{n\geq 1}$ is tight in $\DD([0,T],\H)$. Therefore, by Prokhorov's theorem, it is relatively compact and then we can extract, from $(\widebar\eta^n)_{n\geq 1}$, a subsequence that converges weakly to a limit $(\widebar\eta_t)_{t\in [0,T]} = (\eta_t, R_t)_{t\in [0,T]} \in\DD([0,T],\H)$. We want to prove, in this step, that this limit is unique and defined by (\ref{eq.limite.tcl}). Then, the theorem will follow (see for example \cite[Corollary p.59]{billingsley1968a}).
For a better simplicity in the notations, we assume, without loss of generality that the entire sequence $(\widebar\eta^n)_{n\geq 1}$ converges towards the limit $\widebar \eta=(\eta,R)$.

\bigskip

\begin{lemma}[Convergence of the martingale part]
\label{cv.suite.martingales}
The sequence of martingale processes $(M^n)_{n}$ converges in distribution in $\DD([0,T],W_3^*)$ towards a process $\gau$ with values in $\C([0,T],\C^{0,*}(\X)) \subset \DD([0,T],W_3^*)$, where $\C^{0,*}(\X)$ is the dual of $\C^{0}(\X)$. For any $f \in \C^0(\X)$, the process $\ps{\gau}{f}$ is a continuous centred Gaussian martingale with values in $\RR$ with quadratic variation defined by \eqref{crochet.martin}.
\end{lemma}

\begin{proof}
Let $f\in \C^{0}(\X)$ and $\crochet{\ps{\gau}{f}}_t$ be the quadratic variation defined by \eqref{crochet.martin}, then by \eqref{var_qua},
\begin{align*}
&|\crochet{\ps{M^n}{f}}_t-\crochet{\ps{\gau}{f}}_t|
\\
	&\quad \leq
		\int_0^t \intX \biggl[
		 \taudiv(S^n_u, x) \, 
    		 \int_0^1 
        		\left[ f(\alpha \, x)+f((1-\alpha) \, x)-f(x) \right]^2 \, 
			  Q(\dif \alpha)
			  +D\, f(x)^2 \biggr]\,
\\
	&\qquad \qquad\qquad\qquad
		 \times |\nub_u^n(\dif x)-\xi_u(\dif x)| \, \dif u
\\
	& \qquad +
		\int_0^t \intX |\taudiv(S^n_u, x)-\taudiv(S_u, x)| \, 
     	\int_0^1 
        	\left[ f(\alpha \, x)+f((1-\alpha) \, x)-f(x) \right]^2 \, 
			  Q(\dif \alpha) \,\xi_u(\dif x) \, \dif u  
\\
	& \quad \leq
		(9\,\widebar \taudiv +D)\,\norme{f}^2_\infty \,
		 \int_0^t \intX |\nub_u^n(\dif x)-\xi_u(\dif x)| \, \dif u
		+ 9\,K_{gb}\,\norme{f}^2_\infty C_T^\xi \,
			\int_0^t |S_u^n-S_u|\,\dif u
\end{align*}
with $C_T^\xi$ defined by \eqref{sup.xi}. Therefore, by \eqref{eq:cvEID} and the dominated convergence theorem, $\crochet{\ps{M^n}{f}}$ converges in distribution towards $\crochet{\ps{G}{f}}$.

Moreover, a discontinuity of $t \mapsto \nu_t$ only happens during a birth or death event and the jump of the population number is $\pm 1$. Then from \eqref{def.nubar} and \eqref{renormalisation},  for any $f\in \C^0(\X)$, $|\Delta Z^n_t(f)|\leq \frac{\norme{f}_\infty}{n}$. Therefore, from \eqref{def.Mn}
\begin{align}
\label{maj.delta.M}
	\sup_{t\in[0,T]} |\Delta \ps{M_t^n}{f}|\leq \frac{\Vert f \Vert_\infty}{\sqrt{n}}\,,
\end{align}
and then $\sup_{t\in[0,T]} |\Delta \ps{M_t^n}{f}|$ converges in probability towards 0. Hence, according to \cite[Theorem 3.11 page 473]{jacod2003a}, for each $f$, the sequence of processes $((\ps{M^n_t}{f})_{t \in [0,T]})_{n\geq 1}$ converges to $(\ps{G_t}{f})_{t \in [0,T]}$. To have an (infinite dimensional) convergence of the sequence of (operator valued) processes $((M^n_t)_{t \in [0,T]})_{n\geq 1}$, it then suffices to prove its tightness in $\DD([0,T],W_3^*)$. To do it, it is enough to use \eqref{eq:borne-martinagle} and arguments of the step 1. 
\end{proof}

\begin{lemma}[Limit equation for $(R_t)$]
\label{prop.identification.Rt}
The limit process $(R_t)_{t\in[0,T]}$ satisfies
\begin{align*}
R_t & = 
		-\int_{0}^t\biggl[
			D\,R_u+
			\frac kV \, \ps{\eta_u}{\rhog(S_u,.)}
			+ R_u \, \frac kV \, \ps{\xi_u}{\partial_s\rhog(S_u,.)} 
		\biggr]\,\rmd u\,.
\end{align*}
\end{lemma}

\begin{proof}
By definition, $(R_t)_{t\in[0,T]}$ is a limit point of the sequence of processes $((R^n_t)_{t\in[0,T]})_n$, then by \eqref{def.Rnt}, we have the following limit in distribution : for any $t\in [0,T]$,
\begin{align}
\label{def.R.as.lim.int}
R_t & =
	-\lim_{n \to \infty}
	\int_{0}^t\biggl[
			D \, R^n_u
			+ \frac kV \, \ps{\eta^n_u}{\rhog(S^n_u,.)}
			+ \sqrt{n} \, \frac kV \, \ps{\xi_u}{g(S^n_u,.)-g(S_u,.)}
	\biggr]\,\rmd u\,.
\end{align}
By definition of $\widebar \eta=(\eta,R)$ as a limit of $(\widebar \eta^n)_n$ and as the function $s \mapsto g(s,.)$ is continuous, from Lemma \ref{esp.fini}, $D \, R^n_u+ \frac kV \, \ps{\eta^n_u}{\rhog(S^n_u,.)}$ converges in distribution towards $D \, R_u+ \frac kV \, \ps{\eta_u}{\rhog(S_u,.)}$.
Moreover, for any $x\in\X$, by definition of $R_u^n$ (see \eqref{eq:fluctuation}),
\begin{align*}
\sqrt{n}\,(g(S_u^n,x)-g(S_u,x))
	&=
		\sqrt{n}\,\left(g\left(S_u+\frac{1}{\sqrt{n}}\,R_u^n,x\right)-g(S_u,x)\right)
\\
	&=
		\int_0^1 
			\partial_s g\left(S_u+\frac{\alpha}{\sqrt{n}}\,R_u^n,x\right)\,
			R_u^n\,\dif \alpha\,.
\end{align*}
Therefore, by Assumption \ref{hyp.model}
\begin{align*}
& \left|\sqrt{n} \,\left(\rhog(S_u^n,.)-\rhog(S_u,.)\right)
					-R_u\,\partial_s\rhog(S_u,.)\right|
\\
&\qquad=
	\left|\partial_s\rhog(S_u,.) (R_u^n-R_u)
	+
	R_u^n\,\int_0^1 
			\left[
			\partial_s 
			g\left(S_u+\frac{\alpha}{\sqrt{n}}\,R_u^n,.\right)
			-\partial_s\rhog(S_u,.)
			\right]\,\dif \alpha\right|
\\
&\qquad\leq
	\left|\partial_s\rhog(S_u,.) (R_u^n-R_u)\right|
	+
	\frac{K_{gb}}{2\,\sqrt{n}}\,(R_u^n)^2
\end{align*}
which converges towards 0, in distribution, by Lemma \ref{esp.fini}.
Hence, from Lemma \ref{esp.fini} again and the dominated convergence theorem in \eqref{def.R.as.lim.int}, the conclusion follows.
\end{proof}

\begin{lemma}[Semi-martingale decomposition]
The process $(M_t)_{t\in[0,T]}$ defined for any $f\in W_3$ by
\begin{align}
\nonumber
\ps{M_t}{f}
	&=
	\ps{\eta_t}{f}-\ps{\eta_0}{f}
	-\int_0^t \int_\X
		\biggl[
			\taudiv(S_u,x)\,
					\int_0^1[f(\alpha\,x)+f((1-\alpha)\,x)-f(x)]
					\,Q(\dif \alpha)
\\
\nonumber
	& \qquad \qquad \qquad \qquad\qquad\qquad
		-D\,f(x)
		+g(S_u,x)\,f'(x)
		\biggr]\,\eta_u(\dif x)\,\dif u	
\\
\nonumber
	&\quad 
	-
	 \int_0^t R_u\,
	 	\int_\X
	 		\biggl[
					\partial_s\taudiv(S_u,x)\,
					\int_0^1[f(\alpha\,x)+f((1-\alpha)\,x)-f(x)]
					\,Q(\dif \alpha)
\\
\label{def.Mt}
	& \qquad\qquad \qquad\quad
		+\partial_s\rhog(S_u,x)\,f'(x)
		\biggr]\,\xi_u(\dif x)\,\dif u
\end{align}
has the same law as the process $\gau$ defined in Lemma \ref{cv.suite.martingales}.
\end{lemma}

\begin{proof}
We define, for any $\zeta \in \DD([0,T],W_3^*)$, $f\in W_3$, $t\in[0,T]$,
\begin{align}
\nonumber
\Psi^f_t(\zeta)
	&=
	\ps{\zeta_t}{f}-\ps{\zeta_0}{f}
	-\int_0^t \int_\X
		\biggl[ 
		\taudiv(S_u,x)\,
			\int_0^1[f(\alpha\,x)+f((1-\alpha)\,x)-f(x)]
					\,Q(\dif \alpha)
\\
\nonumber
	& \qquad \qquad \qquad \qquad \qquad \qquad
			-D\,f(x)
			+g(S_u,x)\,f'(x)
			\biggr]\,\zeta_u(\dif x)\,\dif u
\\
\nonumber
	&\quad
	 	- \int_0^t R^\zeta_u\, 
		\int_\X\biggl[
					\partial_s\taudiv(S_u,x)\,
					\int_0^1[f(\alpha\,x)+f((1-\alpha)\,x)-f(x)]
					\,Q(\dif \alpha)
\\
\label{def.psi.f.t}
	&\qquad\qquad\qquad\quad
	+\partial_s\rhog(S_u,x)\,f'(x)
		\biggr]\,\xi_u(\dif x)\,\dif u
\end{align}
where
\begin{align}
\label{eq.R.zeta}
R^\zeta_t = 
	\int_{0}^t\biggl[
		- D \, R^\zeta_u
		- \frac kV \, \ps{\zeta_u}{\rhog(S_u,.)}
		- R^\zeta_u \, \frac kV \, \ps{\xi_u}{\partial_s g(S_u,.)}
	\biggr]\,\rmd u\,.
\end{align}

Following, for example, the approach of \cite[Lemma 5.8]{campillo2014a}, we can prove that $\zeta \mapsto \Psi^f_t(\zeta)$ is continuous from $\mathbb{D}([0,T],W_3^*)$ to $\mathbb{R}$ in any point $\zeta \in \C([0,T], W_3^*)$.
Indeed, using some rough bounds, we have the existence of a constant $C>0$, such that 
$$
|\Psi^f_t(\zeta) - \Psi^f_t(\widetilde{\zeta})| \leq C \sup_{s \in [0,T]} \left\| \zeta_s - \widetilde{\zeta}_s \right\|_{W_3^*}
$$
and on continuous points, the Skorohod topology coincides with the uniform topology. However, as for \eqref{maj.delta.M},
$$
	|\ps{\eta_t^n}{f}-\ps{\eta^n_{t-}}{f}|\leq \frac{\norme{f}_\infty}{\sqrt{n}}
$$
therefore, $\eta$ is a continuous process and then $\lim_{n \to \infty} \Psi_t^f(\eta^n)=\Psi_t^f(\eta)$ in distribution.
By Lemma \ref{cv.suite.martingales}, it is sufficient to prove the proposition that $(\Psi^f_t(\eta^n))_t$ and $\ps{M^n_t}{f}$ converge in distribution towards the same limit.

As $\ps{\eta^n_t}{f} = \ps{A^n_t}{f} + \ps{M^n_t}{f}$, by \eqref{def.Ant} and \eqref{def.psi.f.t},
\begin{align*}
\Psi_t^f(\eta^n)-{M_t^n}{(f)}
	&= B_t^n(f) + C_t^n(f)
\end{align*}
with
\begin{align*}
B_t^n(f)
	&=	 
   	\int_0^t 
	 		\ps{\xi_u}{ 
	 			\left[
	 				\sqrt{n} \,\left(\rhog(S_u^n,.)-\rhog(S_u,.)\right)
					-R_u^{\eta^n}\,\partial_s\rhog(S_u,.)
	 			\right]\,f'} 
	 		\, \dif u 
\\
	&\quad
	+
	\int_{0}^t
	 \intX
	 \left[
	 	\sqrt{n} \,\left(\taudiv(S^n_u,x)-\taudiv(S_u, x)\right)
	 	- R_u^{\eta^n}\,\partial_s\taudiv(S_u, x)
	 \right]\,
\\
	&\qquad \qquad  \qquad
	 \int_0^1
	 	\Bigl[f(\alpha\, x)+ f((1-\alpha)\,x)-f(x)\Bigr] \, 	 			
			Q(\dif \alpha) \,\xi_u(\dif x) \,\dif u 
\end{align*}
and
\begin{align*}
C_t^n(f)
	&=	
		\int_0^t 
		\biggl[
	 		\ps{\eta_u^n}{(\rhog(S_u^n,.)-\rhog(S_u,.))\,f'} 
\\
\nonumber 
  &\quad
  + \intX
  	(\taudiv(S^n_u,x)-\taudiv(S_u,x)) \, 
	 \int_0^1  
			  \bigl[ f(\alpha \, x)+f((1-\alpha)\,x)-f(x)\bigr] \, 
			  Q(\dif \alpha)\,\eta_u^n(\dif x) \, \biggr] \dif u \,.
\end{align*}
By the same approach using in the proof of Lemma \ref{prop.identification.Rt}, we get
\begin{align}
\label{maj.sqrtng}
\left|\sqrt{n} \,\left(\rhog(S_u^n,.)-\rhog(S_u,.)\right)
					-R_u^{\eta^n}\,\partial_s\rhog(S_u,.)\right|
& \leq
	\left|\partial_s\rhog(S_u,.) (R_u^n-R_u^{\eta^n})\right|
	+
	\frac{K_{gb}}{2\,\sqrt{n}}\,(R_u^n)^2
\end{align}
and
\begin{align*}
 \left|\sqrt{n} \,\left(\taudiv(S^n_u,x)-\taudiv(S_u, x)\right)
	 	- R_u^{\eta^n}\,\partial_s\taudiv(S_u, x)\right|
\leq
	\left|\partial_s\taudiv(S_u,x) (R_u^n-R_u^{\eta^n})\right|
	+
	\frac{K_{gb}}{2\,\sqrt{n}}\,(R_u^n)^2\,.
\end{align*}
Hence,

\begin{align*}
\sup_{t\leq T}|B_t^n(f)|
&\leq
	 (\norme{f'}_\infty\,\norme{\partial_s g}_\infty
	 +3\,\norme{f}_\infty\,\norme{\partial_s b}_\infty)\,C_T^\xi\,T\,
	 \sup_{t\leq T}|R_t^n-R_t^{\eta^n}|
\\
	&\qquad
	+(\norme{f'}_\infty+3\,\norme{f}_\infty)\,K_{gb}\,
	\frac{C_T^\xi\,T}{2\,\sqrt{n}}\,\sup_{t\leq T}(R_t^n)^2,
\end{align*}
where $C^\xi_T$ was defined in \eqref{sup.xi}.
From Lemma \ref{esp.fini}, the second term converges towards 0 in probability. 
By \eqref{def.Rnt}, \eqref{eq.R.zeta} and \eqref{maj.sqrtng}

\begin{align*}
|R_t^n-R_t^{\eta^n}|
	&\leq	
	\left(D + \frac{k}{V}\,\norme{\partial_s g}_\infty\,C^\xi_T \right)\, 
	\int_0^t |R_u^n-R_u^{\eta^n}|
	+ K_{gb}\,\frac{k}{V}\,\int_0^t |S_u^n-S_u|\,\ps{\eta_u^n}{1}\,\dif u
\\
	&\quad
	+ \frac{k}{V}\,\frac{K_{gb}}{2\,\sqrt{n}}\, T\,C^\xi_T\,\sup_{t\leq T}|R_t^n|^2  \,.
\end{align*}

From Lemma \ref{esp.fini} and by the Gronwall lemma, we deduce that $(R_t^n-R_t^{\eta^n})_{t\leq T}$ converges uniformly towards 0 in probability and then that $(B_t^n(f))_{t\leq T}$ converges uniformly towards 0 in probability. Furthermore,
\begin{align*}
\sup_{t\leq T}|C_t^n(f)|
	& \leq
		(\norme{f'}_{\infty}+3\,\norme{f}_\infty)\,\frac{K_{gb}\,T}{C_1'}\,
		\sup_{t\leq T}\norme{\eta_t^n}_{W_2^*} 
		\,\sup_{t\leq T}\left|S_t^n-S_t\right|\,
\end{align*}
where that $C_1'$ was defined in \eqref{ineg.normes}.
The sequence $S^n$ converges in distribution towards $S$ then by Lemma \ref{esp.fini}, we deduce that $\sup_{t\leq T}|C_t^n(f)|$ converges in probability towards $0$. Finally, $(\Psi^f_t(\eta^n))_t$ and $\ps{M^n_t}{f}$ have the same limit $\ps{G}{f}$ in distribution. 
\end{proof}

To conclude the proof of Theorem \ref{th:tcl-intro}, it rests to prove the uniqueness of the solution of \eqref{def.Mt} but it can be easily proved via the classic argument involving Gronwall lemma.

\subsection{Proof of Theorem \ref{th:tcl-CY}}
\label{sec.proof.th:tcl-CY}

The proof is quite similar to the proof of Theorem \ref{th:tcl-intro} so we do not provide all details. 

Firstly, similarly to Lemma \ref{esp.fini}, we can use Lemma \ref{lem:martinagle} (with $f\equiv 1$), Doob's inequality, Gronwall lemma and some rough bounds to show that
\begin{equation}
\label{eq:CYmoment}
	\sup_{n\geq 1}\EE\left(\sup_{t\leq T} |N^n_t|^2 + |S^n_t|^2 + |Q_t^n|^2+|R^n_t|^2 \right) < +\infty\,.
\end{equation}
Indeed, note that one can bound $\mu(S^n_s)$ by $\widebar \mu = \max_{0\leq s \leq \Sin \vee S_0} \mu(s)$ for every $s\geq 0$, because $S^n$ remains in $[0, \Sin \vee S_0]$; see for instance \cite[Proposition 2.1]{collet2013}.

Using Equation \eqref{eq:CYmoment} and Markov inequality, we obtain, as in the proof of Theorem \ref{th:tcl-intro} that $(N^n, S^n, Q^n, R^n)_{n\geq 1}$  satisfies the Aldous Robolledo criterion \cite[Corollary 2.3.3.]{joffe1986a} and then that $(N^n, S^n, Q^n, R^n)_{n\geq 1}$ is tight in $\DD([0,T],\RR_+ \times [0, \Sin \vee S_0] \times \RR^2)$.

It remains to show the uniqueness of the limit point. 
Using Lemma \ref{lem:martinagle}, we can see that each limit point of the sequence $(N^n,S^n)$ is solution to the classic chemostat ODE ($i.e.$ the two first equations of \eqref{eq:CYEDS}) and then by uniqueness of the solution, it converges to $(N,S)$. Let now study a limit of a convergent subsequence of $(Q^n,R^n)_n$. Following the way of Lemma \ref{cv.suite.martingales} with $f\equiv 1$, we obtain that 
$
  \sup_{t\in[0,T]} |\Delta \ps{M_t^n}{1}|\leq \frac{1}{\sqrt{n}}\,.
$
Moreover, we have the following convergence in distribution,
$$
	\lim_{n\to\infty}\crochet{\ps{M^n}{1}}_{t\in[0,T]}
	=
	\left(\int_0^t (\mu(S_u)+D) \, N_u \, \dif u\right)_{t\in[0,T]}\,.
$$ 
Then, by \cite[Theorem 3.11 page 473]{jacod2003a}, we deduce the convergence, in distribution, of $((\ps{M_t^n}{1})_{t\in[0,T]})_{n\geq 1}$ towards $(\int_0^t \sqrt{(\mu(S_u)+D)\,N_u}\,\dif B_u)_{t\in[0,T]}$.
The end of the proof is then as in the proof of Theorem \ref{th:tcl-intro}.

\begin{remark}[Infinite dimensional case when $M= \infty$]
\label{rq:extension}
According to the proof of Theorem \ref{th:tcl-CY}, we see that to obtain the convergence from finite $M$ to infinite $M$ ($i.e.$ non compact support for the mass) for the finite-dimensional process then it is enough to prove that the uniform bound of Lemma \ref{esp.fini} remains valid (that is equation \ref{eq:CYmoment}).

The situation is more tricky in infinite dimension. Indeed, firstly, we crucially need the convergence of $\widebar{\nu}^n$ to $\xi$ in the space of positive measure endowed with the weak topology in the proof of Theorem \ref{th:tcl-intro}. \cite{campillo2014a} proved the convergence in the space of positive measure endowed with the vague topology. Although vague and weak topologies coincide on a compact space, it is not valid anymore on non-compact sets. Extending the convergence to the weak topology is not trivial; see for instance \cite{C11,MT12}. Also, note that Inequalities \eqref{ineg.normes} are also no longer valid in infinite dimension.
\end{remark}

\section{The Crump-Young model}
\label{sect:crump-young}

In this section, we propose an application of the previously demonstrated central limit theorem (Theorem \ref{th:tcl-CY}) to understand the Crump-Young model. In particular, Section \ref{sect:proofCY} contains the proof of Theorem \ref{th:intro-CY}, which gives an approximation of the long-time behavior of the Crump-Young model (see \eqref{invariant.mesure2}). The process $( S^n_t,N^n_t)_{t\geq 0}$ satisfies the Markov property and is generated by the following infinitesimal generator
\begin{align*}
\mathcal{L} f(s,\ell) 
&= \left[D(\Sin-s) - \frac{km}{nV} \mu(s)\,\ell\right] \partial_s f(s,\ell) + \mu(s)\, \ell\, \left(f(s,\ell+1)-f(s,\ell)\right)\\
&\qquad + D \,\ell\,\left(f(s,\ell-1)-f(s,\ell)\right),
\end{align*}
for all $\ell\geq 0, s\geq 0$ and smooth $f$. This model is a particular case of the general model of \cite{campillo2014a}, where we suppose that division rate and the growth rate (\textit{per capita}) do not depend on the mass of the bacteria. This is a rough assumption which enables us to considerably weaken the dimension of the problem (from an infinite dimension to two dimensions). Our main result implies that it can be approximated by \eqref{SDE} and this diffusion process will be the main object of interest. Note that, through the function $\mu$, we introduce a parameter $m$ which can be understood as the mean size of one bacterium induced by the mass-structured model. Indeed, if we consider the integro-differential equation \eqref{eq.limite.eid.faible} with parameters given by \eqref{eq:parameterCY} and we set 
$$
\forall t\geq0, \ N_t= \int_\X \xi_t(dx), \quad Y_t= \int_\X  x \ \xi_t(dx),
$$
where $N_t$ represents the number of individuals at time $t$ and $Y_t$ the biomass. As pointed out by \cite[Section 5.4]{campillo2014a}, one can prove that these two quantities can be described as a solution to the classic chemostat equations. Moreover 
\begin{align*}
\frac{\dif}{\dif t} \frac{Y_t}{N_t}
	&=
		\mu(S_t)\,\left(m-\frac{Y_t}{N_t}\right)\,,
\end{align*}
and then $\frac{Y_t}{N_t}$  converges to $m$ when $t$ tends to infinity. Before studying rigorously the behavior of the system \eqref{eq:CYEDS}, let us end this section by a remark on the modelling.

\begin{remark}[Reinforced process for indirect interactions]
Consider the system \eqref{eq:CYEDS}, with starting points $N_0 = N^*$, $S_0=S^*$, $R_0=0$, $Q_0\in \mathbb{R}$, where $(N^*,S^*)\neq (0,\Sin)$ is some equilibrium of the two first equations. As $\mu(S^*)=D$, the system then reduces to
\begin{equation*}
     \begin{cases}
\dif Q_t 	&= \mu'(S^*)\, R_t \, N^* \, \dif t	+ \sqrt{2\,D\,N^*}\,\dif B_t, \\
\dif R_t &= -\left[	D\,R_t+	\frac kV \, \mu(S^*)\,m\,Q_t	+ \frac kV \,R_t \, \mu'(S^*)\,m\,  N^* \right] \dif t.
     \end{cases}
\end{equation*}
In particular the second equation became a simple linear (ordinary) differential equation and then by the variation of constants method, we have
$$
	R_t = 	- \frac{k}{V}\,m\,\mu(S^*)
		\int_0^t e^{-(D+\frac{k}{V}\,m\,\mu'(S^*)\,N^*)\,(t-s)}\,Q_s\,\dif s.
$$
Hence
$$
\dif Q_t
	= -\mu'(S^*)\,N^*\,\frac{k}{V}\,m\,\mu(S^*)
		\int_0^t e^{-(D+\frac{k}{V}\,m\,\mu'(S^*)\,N^*)\,(t-s)}\,Q_s\,\dif s
		 \,\dif t
	+\sqrt{2\,D\,N^*}\,\dif B_t\,.
$$ 

The solution of this equation then represents the evolution of the population around an equilibrium under an indirect competition (presence of substrate). This process belongs to the large class of self-interacting diffusions; see \cite{Gadat2014,Gadat2015} and reference therein. These processes are not Markov and if, more generally,
$$
\dif Q_t
	= -\int_0^t \kappa(t-s) \,Q_s\,\dif s
		 \,\dif t
	+\sqrt{2\,D\,N^*}\,\dif B_t\,,
$$
for some function $\kappa$, then $\kappa$ represents the memory of the substrate consumption. In a different context than the chemostat, one can imagine a different function $\kappa$ to model an indirect interaction which can influence the size of the population.
 \end{remark}
 
\subsection{Proof of theorem \ref{th:intro-CY}}
\label{sect:proofCY}

In this section, we study the solution of the system of equations \eqref{eq:CYEDS} under the assumptions of Theorem \ref{th:intro-CY}.
 
Firstly, let us see that the two first equations of \eqref{eq:CYEDS} forms a homogeneous system of ODE. It is the classic chemostat equations; see \cite{smith1995a}. In particular, as the specific growth rate $\mu$ is supposed to be increasing, the couple $(N_t,S_t)_{t\geq 0}$ admits only two equilibria that are $(0,\Sin)$, which is usually called the washout and corresponds to the extinction of the population, and another $(N^*,S^*)$ corresponding to the unique solution of 
$$
\mu(S^*)=D \qquad \text{and} \qquad
N^* = \frac{V }{km} (\Sin -S^*).
$$
Moreover, we have
$$
\dif \left(S_t + \frac{mk}{V} N_t\right) = D \left(\Sin -\left(S_t + \frac{mk}{V} N_t\right)\right) \dif t,
$$
and then
$$
\lim_{t \to \infty} S_t + \frac{mk}{V} N_t = \Sin.
$$
Also a calculus of the Jacobian at these two points shows that $(N^*,S^*)$ is stable while $(0,\Sin)$ is unstable. As a consequence, the Poincaré-Bendixson  theorem (see for instance \cite[Page 9]{smith1995a}) entails that, whatever the initial condition $(N_0,S_0) \in \mathbb{R}_+^* \times \mathbb{R}_+$, the following deterministic convergence holds:
$$
\lim_{t \to \infty} (N_t,S_t)=(N^*,S^*).
$$
Now, let us study the dynamics of $(Q_t,R_t)_{t\geq 0}$. We set $Z_t=(Q_t,R_t)^T$, then
\begin{equation}
\label{eq:OU}
	\dif Z_t = A_t\,Z_t\,\dif t + C_t\, \dif B_t,
\end{equation}
where
$$
A_t =
	\begin{pmatrix}
   		\mu(S_t)-D & \mu'(S_t)\,N_t \\
   		-\frac{k}{V}\,\mu(S_t)\,m	& -(D+\frac{k}{V}\,\mu'(S_t)\,N_t\,m)
	\end{pmatrix},
 \qquad
C_t =
	\begin{pmatrix}
   		\sqrt{(\mu(S_t)+D)\,N_t}\\
   		0
	\end{pmatrix}.
$$
In particular, one can think $(N_t,S_t,Q_t,R_t)_{t\geq 0}$ as a homogeneous-time Markov process or only $(Q_t,R_t)_{t\geq 0}$ as an inhomogeneous one.
Equation \eqref{eq:OU} is now linear, and classically we set $\widetilde Z_t = e^{-\int_0^t A_s\,\dif s}\, Z_t$, to obtain
$\dif \widetilde Z_t =  e^{-\int_0^t A_s\,\dif s}\, C_t\, \dif B_t$ and
\begin{align}
\label{eq.exp.Z}
	Z_t = e^{\int_0^t A_s\,\dif s}\,Z_0
	+
	\int_0^t e^{\int_s^t A_u\,\dif u}\, C_s\,\dif B_s\,.
\end{align}
Therefore, for all $t\geq 0$, the law of $Z_t$ is a Gaussian distribution of mean $e^{\int_0^t A_s\,\dif s}\,\EE(Z_0)$ and variance matrix $\Sigma_t$ given by
$$
	\Sigma_t:= 
		\int_0^t
			e^{\int_s^t A_u\,\dif u}\, C_s\,
			C_s^T\,e^{\int_s^t A^T_u\,\dif u}\, \dif s.
$$
To prove the convergence in law of $Z_t$ to a Gaussian variable, it is then enough to study the convergence of its mean and its variance. Note that the eigenvalues of $A_s$ are
$$
	\lambda^1_s= \mu(S_s)-D-\frac{k}{V}\,m\,\mu'(S_s)\,N_s, \quad \lambda^2_s= -D,
$$
which are (at least for large $s$ because $(N_s,S_s) \to (N^*,S^*)$) negative because $\mu(S^*)=D$ and $\mu'>0$. Nevertheless, it does not directly imply the convergence of the mean; see for instance \cite[Example 2.2]{Amato2006}. However, we have 
$$
A_\infty := \lim_{t \to \infty} A_t = 
\begin{pmatrix}
   		0 & \mu'(S^*)\,N^* \\
   		-\frac{k}{V}\,\mu(S^*)\,m	& -(D+\frac{k}{V}\,\mu'(S^*)\,N^*\,m)
	\end{pmatrix},
$$
whose eigenvalues are $\lambda_+= -\frac{k}{V}\,m\,\mu'(S^*)\,N^*$ and $\lambda_-= -D$, and then, by a Cesàro-type theorem, \cite[Theorem 2.9]{Amato2006}, we have
$$
\lim_{t \to \infty} e^{\int_0^t A_s\,\dif s}=0.
$$
We have then obtained the convergence of the mean, it rests to prove the convergence of the variance matrix to
$$
\Sigma_\infty = \int_0^\infty e^{A_\infty \,u}\,C_\infty \,C_\infty^T\,e^{A^T_\infty\,u}\,\dif u,
$$
where
$$
C_\infty :=
	\begin{pmatrix}
   		\sqrt{2D N^*} \\
   		0
	\end{pmatrix} = \lim_{t \to \infty} C_t\,.
$$

 Again by \cite[Theorem 2.9]{Amato2006}, there exist $C,\alpha>0$ such that
\begin{equation*}
\forall t\geq 0, \ \left\Vert e^{\int_0^t A_s\,\dif s} \right\Vert + \left\Vert e^{\int_0^t A_\infty\,\dif s} \right\Vert \leq C e^{- \alpha t}\,,
\end{equation*}
where $\norme{.}$ is the standard matrix norm.
Hence, there exists a constant $K$ such that for any $\tau$ and $t \geq \tau$,
\begin{align*}
\Vert \Sigma_t - \Sigma_\infty \Vert
&= \left\Vert \int_0^t
			e^{\int_{t-s}^{t} A_u\,\dif u}\, C_{t-s}\,
			C_{t-s}^T\,e^{\int_{t-s}^{t} A^T_u\,\dif u}\, \dif s - \int_0^\infty e^{A_\infty \,s}\,C_\infty \,C_\infty^T\,e^{A^T_\infty\,s}\,\dif s \right\Vert  \\
&\leq  \int_0^\tau  \left\Vert e^{\int_{t-s}^{t} A_u\,\dif u}\, C_{t-s}\,
			C_{t-s}^T\,e^{\int_{t-s}^{t} A^T_u\,\dif u} - e^{A_\infty \,s}\,C_\infty \,C_\infty^T\,e^{-A_\infty\,s} \right\Vert \dif s\\
&\quad + \int_\tau^\infty 
\left\Vert e^{\int_{t-s}^{t} A_u\,\dif u}\, C_{t-s}\,
			C_{t-s}^T\,e^{\int_{t-s}^{t} A^T_u\,\dif u} \right\Vert \dif s \\	
&\quad + \int_\tau^\infty \left\Vert e^{A_\infty \,s}\,C_\infty \,C_\infty^T\,e^{A^T_\infty\,s} \right\Vert \dif s\\
&\leq  \int_0^\tau  \left\Vert e^{\int_{t-s}^{t} A_u\,\dif u}\, C_{t-s}\,
			C_{t-s}^T\,e^{\int_{t-s}^{t} A^T_u\,\dif u} - e^{A_\infty \,s}\,C_\infty \,C_\infty^T\,e^{A^T_\infty\,s} \right\Vert \dif s\\
&\quad + K e^{- 2\,\alpha \,\tau}.
\end{align*}
The introduction of the variable $\tau$ allows us to obtain an integral whose integration interval does not depend on $t$.
As the second term of the last member is negligible for large $\tau$, by dominated convergence, it then remains to prove that the last integrand vanishes when $t\to \infty$. This is a direct application of the convergences of $(A_t)_{t\geq0}$, $(C_t)_{t \geq 0}$ and the continuity of the various applications (exponential, product...). Also note that $\int_{t-s}^{t} A_u \dif u = \int_{0}^{s} A_{t-u} \dif u$. This concludes the proof of the convergence of $(\Sigma_t)_{t\geq 0}$ and then of the convergence of $(N_t,S_t,Q_t, R_t)_{t\geq 0}$. Let us finally express the calculus of $\Sigma_\infty$. 
We have 
$$
e^{A_\infty\,s}
	=
	\frac{1}{(D-L)}\,
	\begin{pmatrix}
   		(D\,e^{-L\,s}-L\,e^{-D\,s})
   		& \mu'(S^*)\,N^*\,(e^{-L\,s}-e^{-D\,s}) \\
   		\frac{k}{V}\,m\,D\,(e^{-D\,s}-e^{-L\,s})	& 
   		D\, e^{-D\,s}-L\,e^{-L\,s}
	\end{pmatrix}
$$
where $L=\frac{k}{V}\,m\,\mu'(S^*)\,N^*=(\Sin-S^*)\,\mu'(S^*)$, and then 
\begin{align*}
&e^{A_\infty\,s}\,C_\infty\,C_\infty^T\,e^{A_\infty^T\,s} 
= \frac{2\,D\,N^*}{(D-L)^2}\, \times 
\\
	&\begin{pmatrix}
   		(D\,e^{-L\,s}-L\,e^{-D\,s})^2
   		& \frac{k}{V	}\,m\,D\,(D\,e^{-L\,s}-L\,e^{-D\,s})(e^{-D\,s}-e^{-L\,s})  \\
   		\frac{k}{V	}\,m\,D\,(D\,e^{-L\,s}-L\,e^{-D\,s})(e^{-D\,s}-e^{-L\,s}) 
   		& \left(\frac{k}{V}\,m\,D\right)^2\,(e^{-D\,s}-e^{-L\,s})^2
	\end{pmatrix}\,.
\end{align*}
As a consequence, the term $\int_0^\infty e^{A_\infty\,s}\,C_\infty\,C_\infty^T\,e^{A_\infty^T\,s} \dif s$ is equal to
\begin{align*}
\frac{2\,D\,N^*}{(D-L)^2}\, \times 
	\begin{pmatrix}
   		\frac{D^2}{2\,L} + \frac{L^2}{2\,D} - \frac{2\,D\,L}{D+L}
   		& \frac{k}{V}\,m\,D
   			\left[1-\frac{D}{2\,L}-\frac{L}{2\,D}\right] \\
   		\frac{k}{V}\,m\,D\,
   			\left[1-\frac{D}{2\,L}-\frac{L}{2\,D}\right]
   		& (\frac{k}{V}\,m\,D)^2\,
   		\left[\frac{1}{2\,D}+\frac{1}{2\,L} - \frac{2}{D+L}\right]
	\end{pmatrix}\,.
\end{align*}
Finally
\begin{align*}
\Sigma_\infty
	&=
	\frac{D\,N^*}{L}\, \times 
	\begin{pmatrix}
   		\frac{L^2+3\,L\,D + D^2}{D\,(D+L)}
   		& - \frac{k}{V}\,m \\
   		- \frac{k}{V}\,m
   		& (\frac{k}{V}\,m)^2\, \frac{D}{(D+L)}
	\end{pmatrix}\\
	&= 
	\begin{pmatrix}
   		\frac{\left(\frac{k}{V}\,m\,\mu'(S^*)\,N^*+\frac 32 \, D\right)^2- \frac 54 \, D^2}
   			{\frac{k}{V}\,m\,\mu'(S^*)\,(D+\frac{k}{V}\,m\,\mu'(S^*)\,N^*)}
   		& - \frac{D}{\mu'(S^*)} \\
   		- \frac{D}{\mu'(S^*)}
   		& \frac{k}{V}\,m\, \frac{D^2}{\mu'(S^*)\,(D+\frac{k}{V}\,m\,\mu'(S^*)\,N^*)}
	\end{pmatrix}.
\end{align*}

\begin{remark}[Rate of convergence]
Due to the simple form of \eqref{eq:OU}, one can give some estimates on the rate of convergence. Let $\mathcal{W}_2$ be the (second order) Wasserstein distance, defined for any probability measure $\mu,\nu$ by
$$
\mathcal{W}_2(\mu,\nu)= \inf \mathbb{E}[\Vert X- Y \Vert^2_2]^{1/2},
$$
where $\Vert \cdot \Vert_2$ is the classic Euclidean norm in $\RR^2$, and the infimum runs over all random vectors $(X,Y)$ with $X\sim \mu$ and $Y\sim \nu$. Using \cite[Proposition 7]{givens1984a}, we find, for any $t\geq 0$,
$$
\mathcal{W}_2\left(\mathcal{L}(Z_t), \mathcal{N}(0,\Sigma_\infty)\right)= \left\Vert e^{\int_0^t A_s ds} \right\Vert_2^2 + \textrm{Tr}\left(\Sigma_t + \Sigma_\infty - 2\,\Sigma_t^{1/2}\, \Sigma_\infty\, \Sigma_t^{1/2} \right),
$$
where $\textrm{Tr}$ is the classic trace operator. The decay of the right-hand side depends on the rate of convergence of the two-component ODE towards $(N^*,S^*)$. However, even if we assume that $(N_0,S_0)=(N^*,S^*)$, one can not simplify this expression because, even in this case, $\Sigma_t= \int_0^t e^{A_\infty\,s}\,C_\infty\,C_\infty^T\,e^{A_\infty^T\,s} \dif s$ does not necessary commute with $\Sigma_\infty$. The bound of \cite[Proposition 7]{givens1984a} also induces a bound in Wasserstein distance for the four-component process $(S_t,N_t,Q_t,R_t)_{t\geq 0}$. This is not trivial because it is not the case, for example, in total variation in contrast to $(Z_t)_{t \geq 0}$.
\end{remark}

\begin{remark}[An example of non-increasing growth rate]

Let us consider the following growth rate:
$$
\mu: s\mapsto \frac{\mu_{\max}\, s}{K +s +s^2/C},
$$
where $\mu_{\max},K,C$ are some positive constants. This rate is often called Haldane kinetics in the literature and can sometimes be more realistic in application; see for instance \cite{mailleret2004a}. We assume that $D> \mu(\Sin)$ and $\sup_{s \in [0,\Sin]} \mu(s)> D$. 

In this case there are two solutions of 
$$
\mu(S)=D\,, \qquad 
N = \frac{V }{km} (\Sin -S)\,,
$$
(see Figure \ref{fig.haldane}). Let us denote by $(0,\Sin)$, $(N^*,S^*)$, $(N^{ue},S^{ue})$ the tree equilibria for $(N_t,S_t)_{t\geq 0}$, with $\mu'(S^*)>0$ and $\mu'(S^{ue})<0$. The study of the Jacobian matrix and the Poincaré-Bendixson theorem implies here that, if the ODE system does not start from the unstable equilibrium $(N^{ue}, S^{ue})$ then it necessary converges to one of the two stable equilibria $(N^*,S^*)$ or $(0,\Sin)$, depending on the initial condition. As a consequence, the set of invariant distributions of the process $(N_t,S_t,Q_t,R_t)_{t\geq 0}$ is the convex hull of the Gaussian distributions \eqref{invariant.mesure} (with the stable equilibrium $(N^*,S^*)$ for the Haldane growth) and $\delta_{(0,\Sin,0,0)}$. Indeed, for the stable equilibrium $(0,\Sin)$, the matrix $A_\infty$ is zero as well as the vector $C_\infty$ (when we replace $(N^*,S^*)$ by $(0,\Sin)$) while $\Sigma_t$ explodes for $(N^{ue},S^{ue})$ because $A_\infty$ admits as positive eigenvalue $-\frac{k}{V} m \mu'(S^{ue})N^{ue}$ (again when we replace $(N^*,S^*)$ by $(N^{ue},S^{ue})$).
Also, mimicking the previous proof gives the convergence to one of them according to the starting distribution.
\end{remark}

\begin{figure}
\begin{center}
\includegraphics[width=7cm]{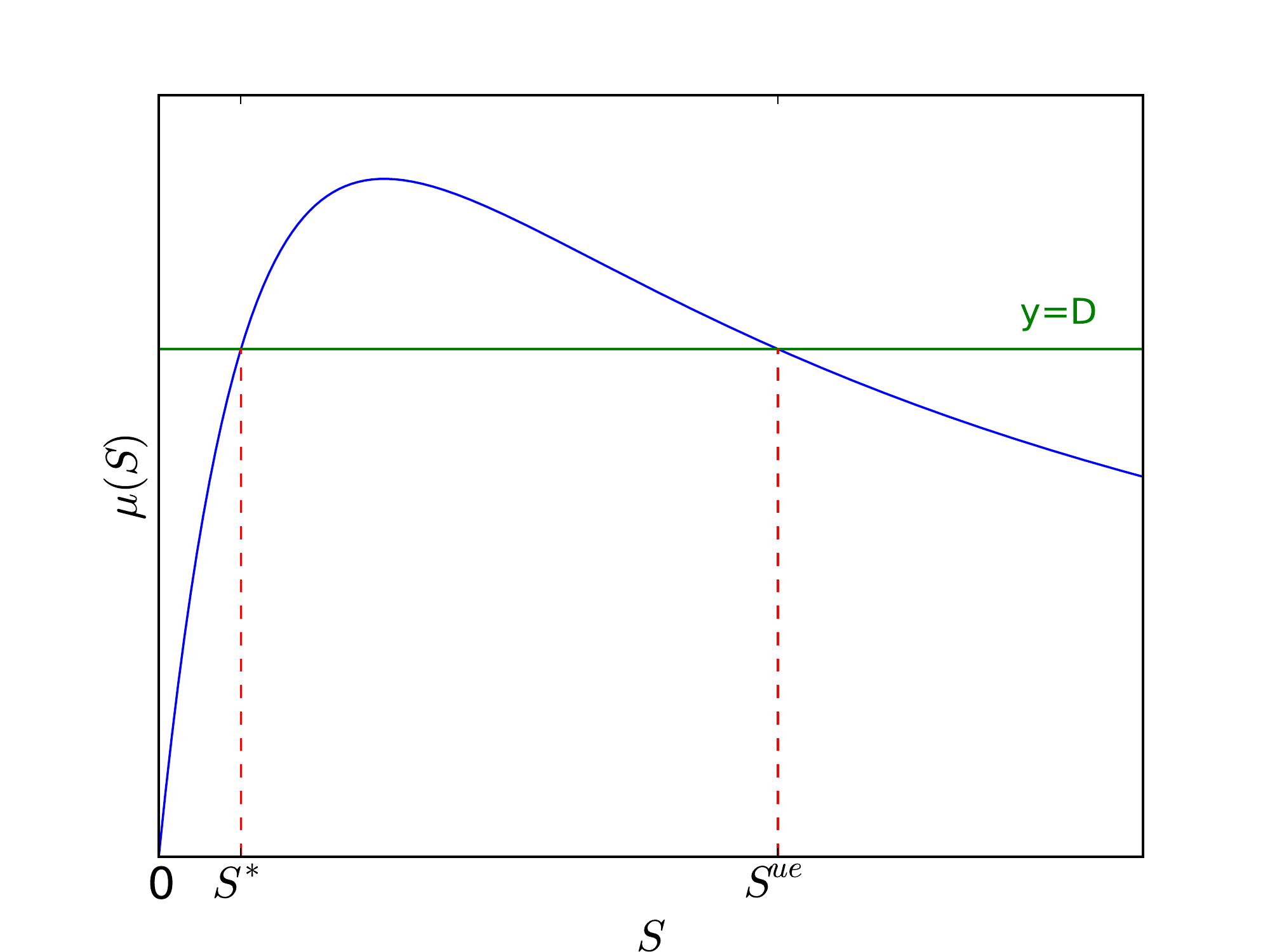}
\end{center}
\caption{Equilibria of the substrate concentration with respect to the Haldane specific growth rate.}
\label{fig.haldane}
\end{figure}

\subsection{Numerical simulations and discussion}
\label{subsec.num.sim}

We use a Gillespie algorithm for the simulation of the Crump-Young model (C-Y) (see Algorithm 1 of \cite{fritsch2015a}) and an Euler method for the simulation of the stochastic differential equations (SDE) \eqref{SDE}. The system of ordinary differential equations (ODE) (two first equations of \eqref{eq:CYEDS}) is solved by the \texttt{odeint} function of the \texttt{scipy.integrate} module of \texttt{Python}. 

\medskip

In general, a chemostat is described by the substrate and the biomass concentrations rather than the substrate concentration and the number of individuals. The biomass concentration is obtained by multiplying the number of individuals by $\frac{m}{V}$, therefore, the graphs are the same up to the multiplicative constant $\frac{m}{V}$.

\subsubsection{Monod growth}
We use the Monod growth parameters of the Escherichia coli bacteria in glucose with a temperature equals to 30 degrees Celsius \citep{monod1942a}, i.e. 
$$
	\mu(S)=1.35\,\frac{S}{0.004+S}\,, \qquad k = 0.23\,,
$$
and with $m = 7\,. 10^{-13}\, g$, $D=0.5$ h$^{-1}$, $S_0=\Sin=0.003$ g.l$^{-1}$

\medskip

\begin{figure*}
\begin{center}
\begin{tabular}{c@{\hskip-0.5em}c@{\hskip-0.5em}c}
\includegraphics[width=5.5cm]{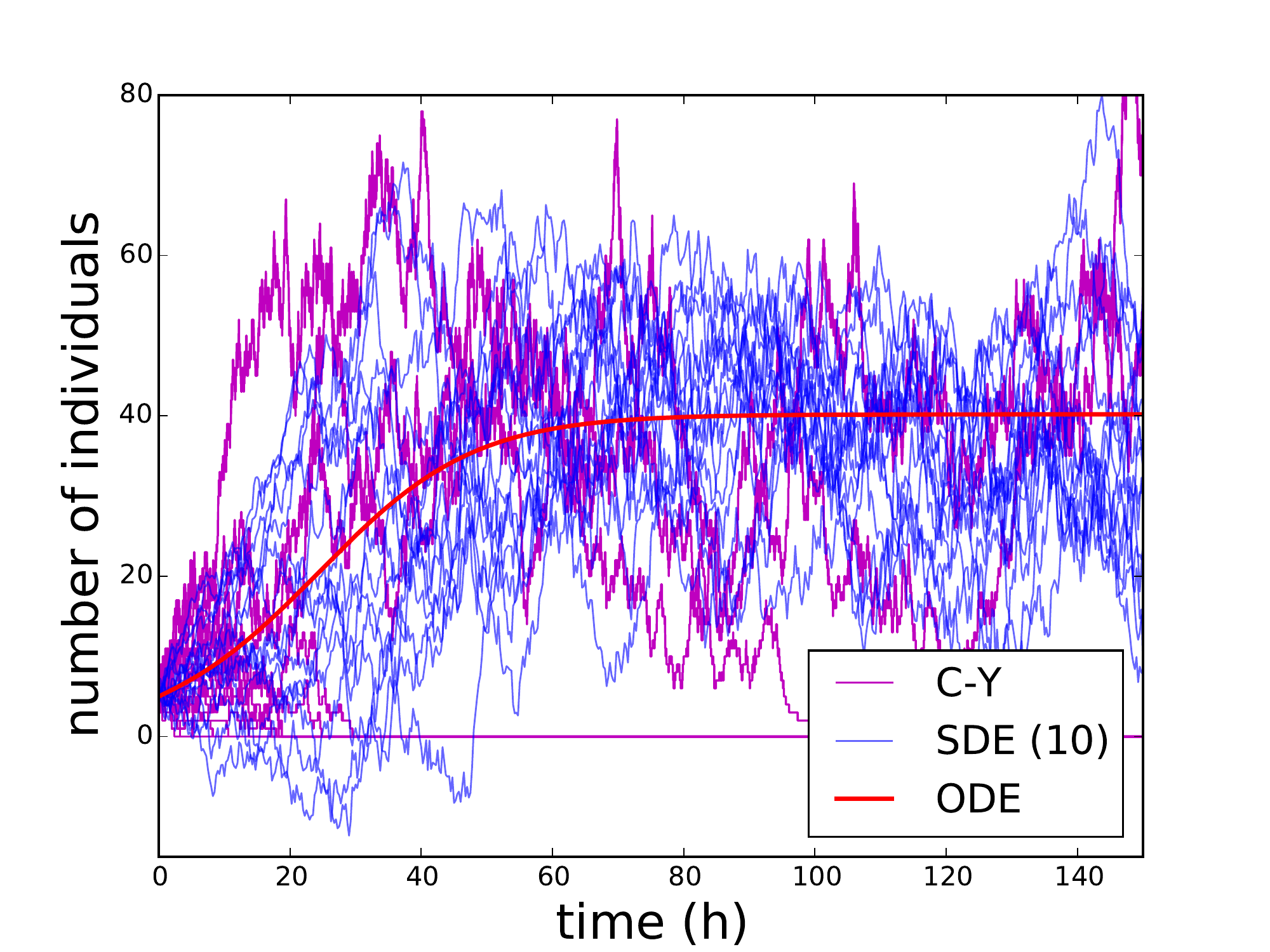}
&
\includegraphics[width=5.5cm]{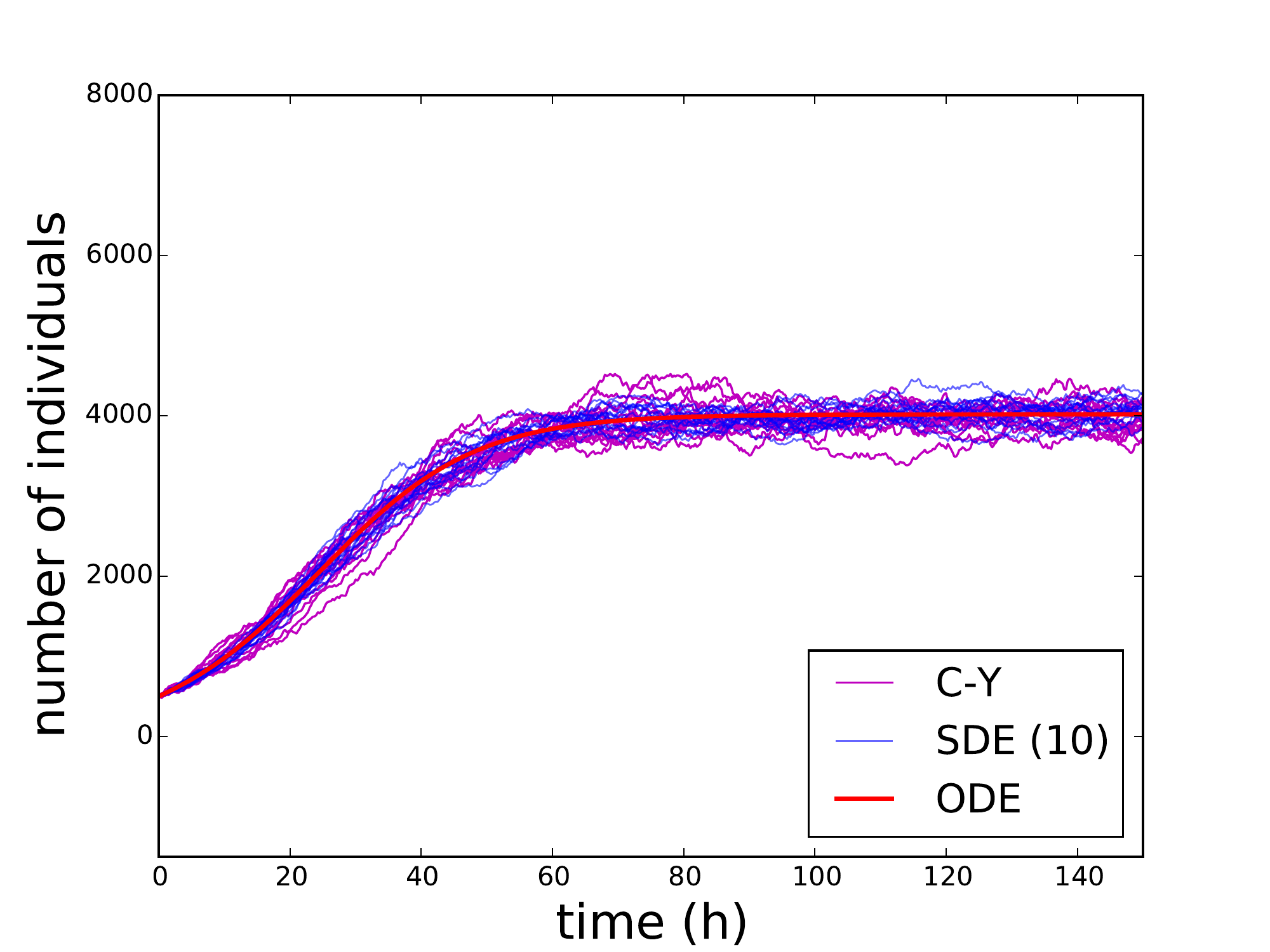}
&
\includegraphics[width=5.5cm]{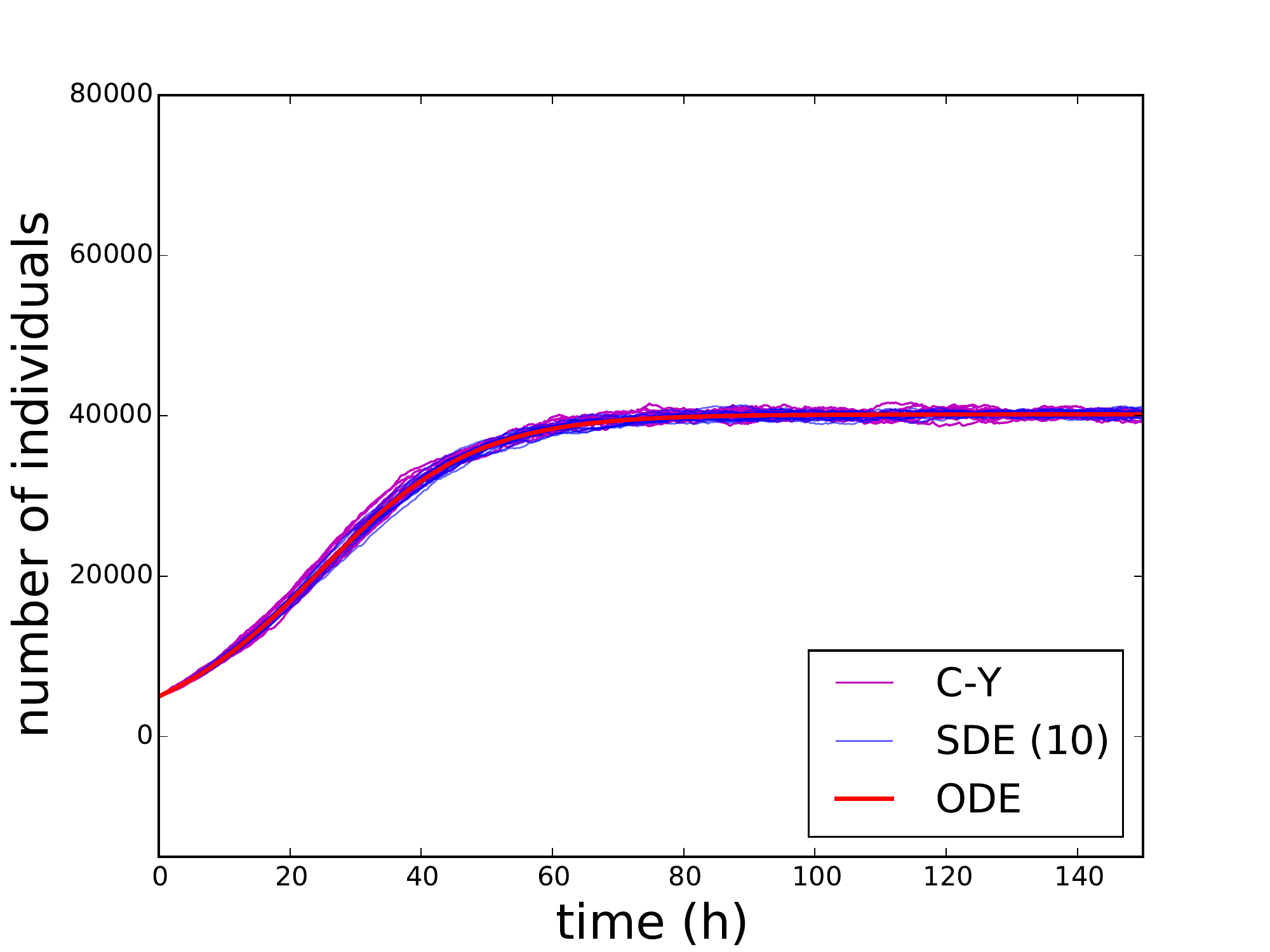}
\\ 
\includegraphics[width=5.5cm]{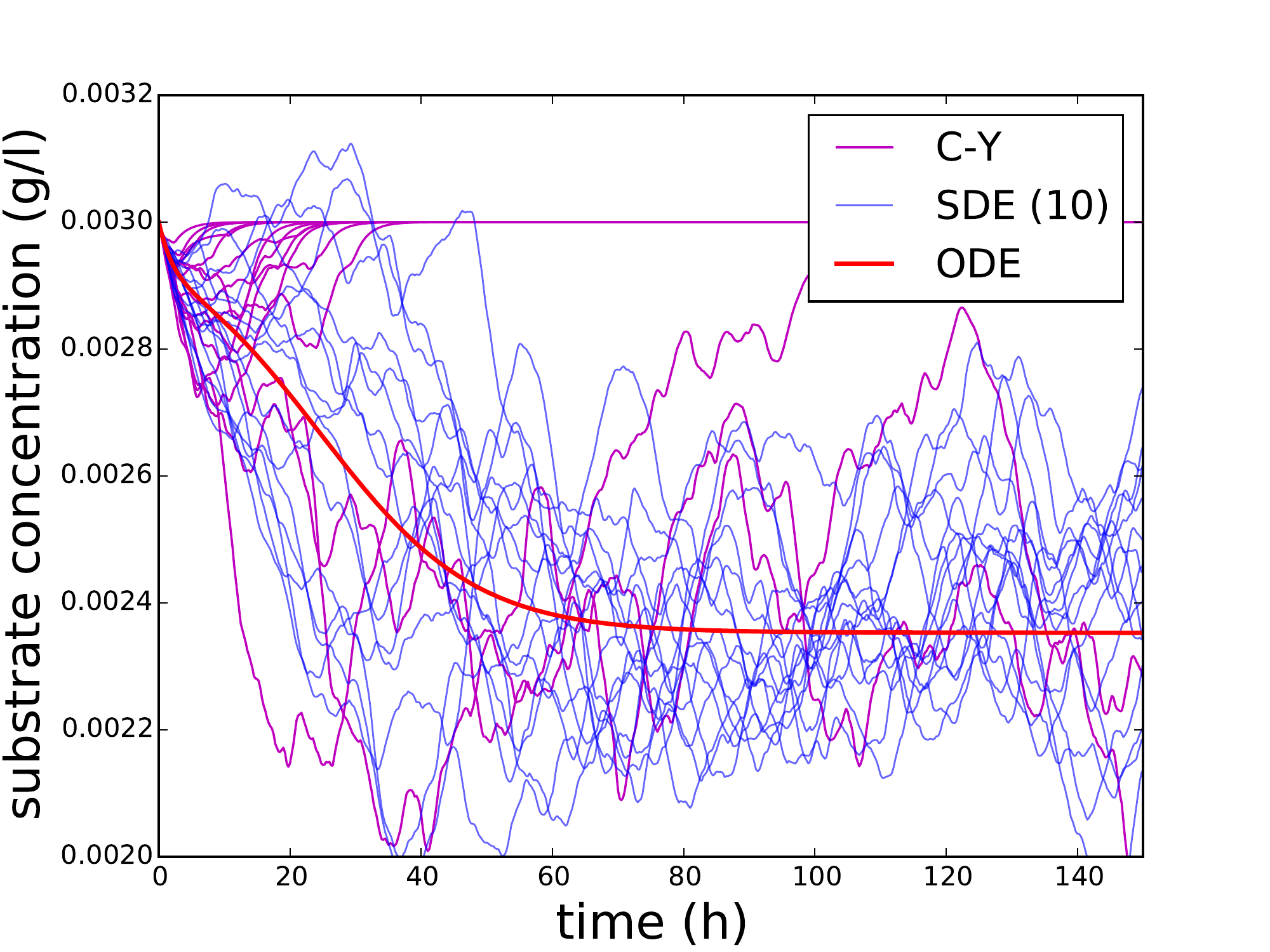}
&
\includegraphics[width=5.5cm]{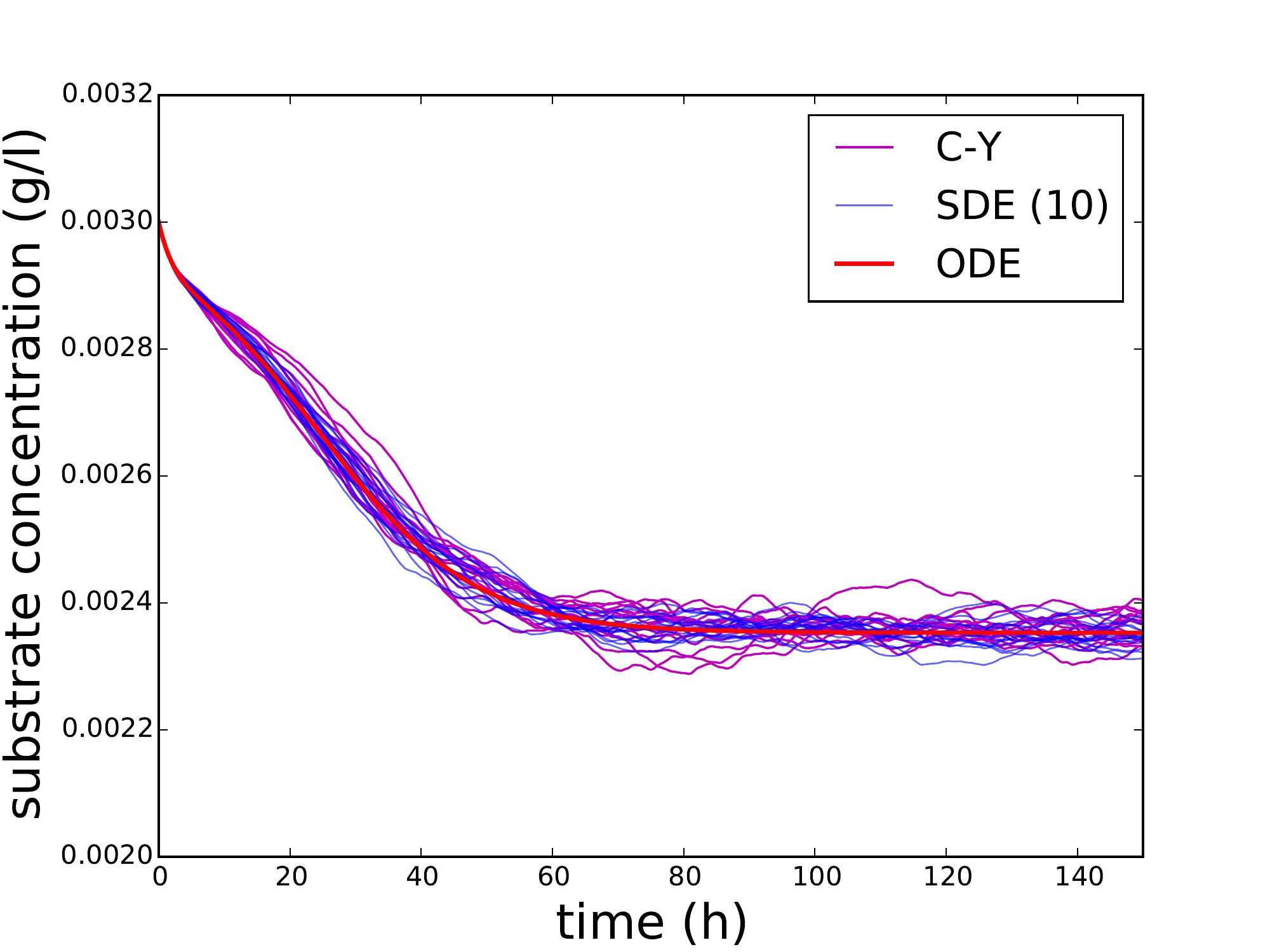}
&
\includegraphics[width=5.5cm]{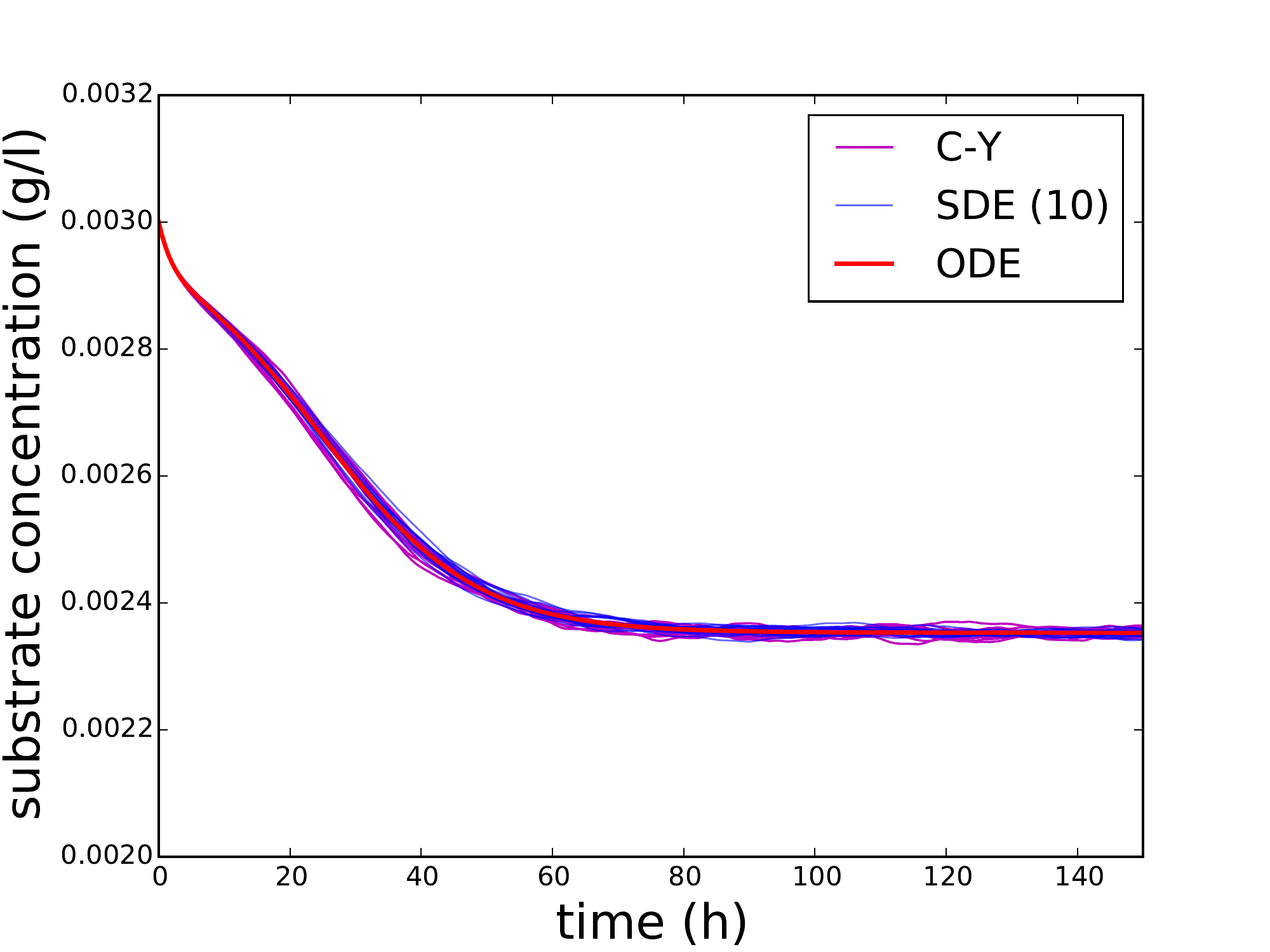}
\\ 
\scriptsize small population size
&
\scriptsize medium population size
&
\scriptsize large population size
\\
\scriptsize $V=10^{-8}$ l , $N_0 = 5$
&
\scriptsize $V=10^{-6}$ l , $N_0 = 500$
&
\scriptsize $V=10^{-5}$ l , $N_0 = 5000$
\end{tabular}
\end{center}
\vskip-1em
\caption{Time evolution of the number of individuals (top) and the substrate concentration (bottom) for small, medium and large population sizes for the Monod growth model. The magenta lines are simulations of 15 independent runs of the Crump-Young model, the blue lines are simulations of 15 independent runs of the system of stochastic differential equations \eqref{SDE} and the large red line is the simulation of the system of ordinary differential equations given by the two first equations of \eqref{eq:CYEDS}.
\label{fig.CV.CY}}
\end{figure*}

\begin{table}
\caption{Simulation times, in seconds, of the simulations of Figure \ref{fig.CV.CY} (15 runs of the Crump-young model and the SDE in small, medium and large population sizes).
Simulations were carried out on a laptop computer with 2.5 GHz i5 (2 cores) processor and 4 GB memory.}
\label{table.CPU.time}
\begin{center}
\begin{tabular}{|c|c|c|c|c|}
	\hline
    Population size & Small & Medium & Large  \\
     \hline
    15 Crump-young	runs		&	3.197	& 578.235	& 4797.546 \\
    15 runs of the SDE		&	4.631	& 4.573		& 4.272 \\	
    \hline
\end{tabular}
\end{center}
\end{table}

\medskip

The convergence, in large population size, of the Crump-Young model towards the SDE \eqref{SDE} is illustrated in Figure \ref{fig.CV.CY}. 
In small population size, the behavior of the Crump-Young model is different from the one of the SDE. In particular, contrary to the Crump-Young model, the SDE can not depict the population extinction. Moreover, in small population size, we observe that the number of individuals can be negative for the SDE, therefore this model is not satisfactory in this situation. Also note that the Crump-Young model is a jump model, whereas the SDE is a continuous model. 
However, in large population size, the jumps of the number of individuals ($\pm 1$) in the Crump-Young model become negligible with respect to the population size, then this model can be approximated by a continuous one. According to Figure \ref{fig.CV.CY}, the SDE seems to be a good approximation of the Crump-Young model from medium population size. 
Moreover it is much faster to compute than the Crump-Young model (see Table \ref{table.CPU.time}). In very large population, both models converge to the deterministic system of ODE, given by the two first equations of \eqref{eq:CYEDS}, then the ODE model is sufficient to describe the behavior of the chemostat in this context.

\begin{figure*}
\begin{center}
\begin{tabular}{c@{\hskip2em}c}
\includegraphics[width=7cm]{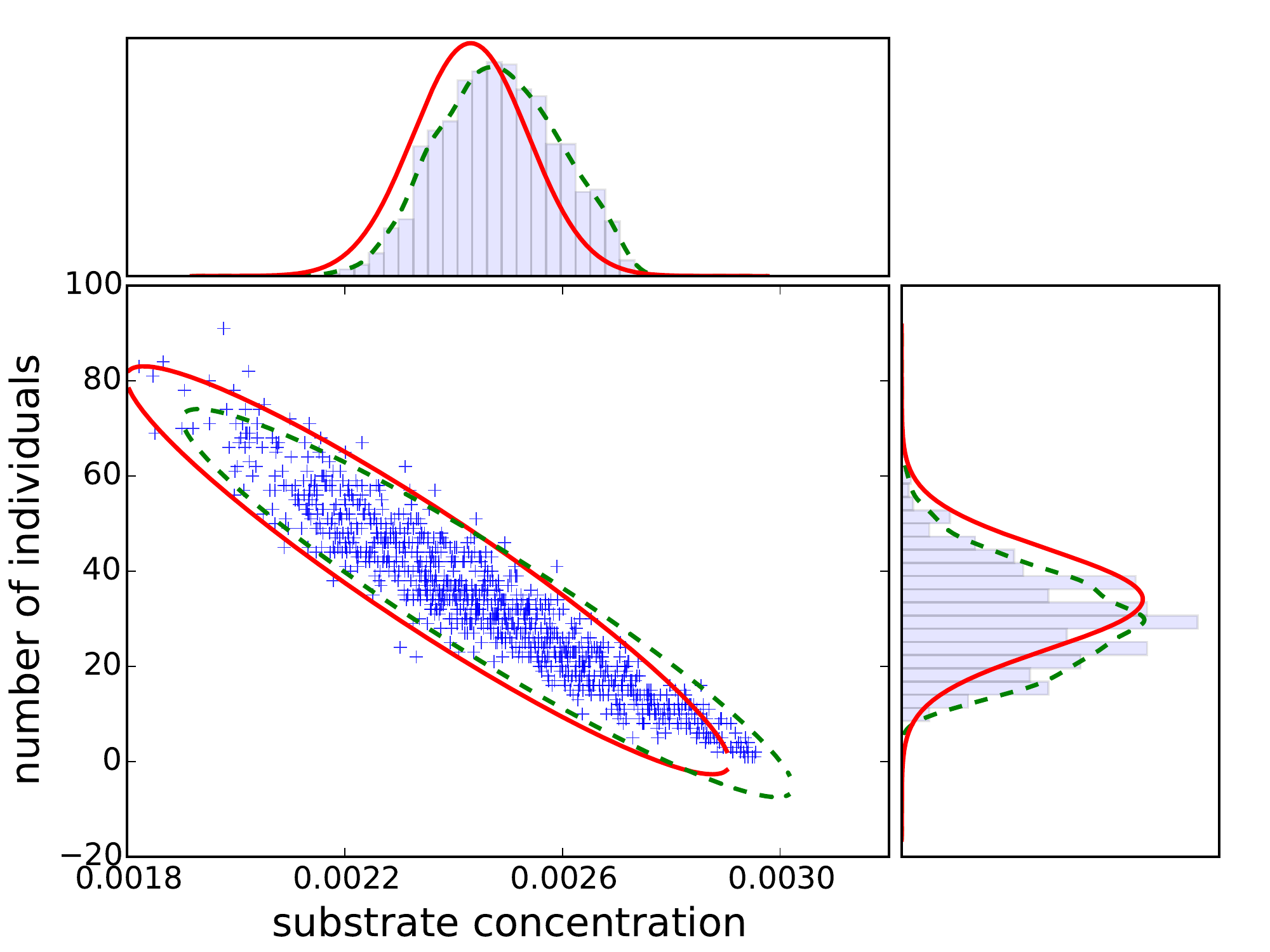}
&
\includegraphics[width=7cm]{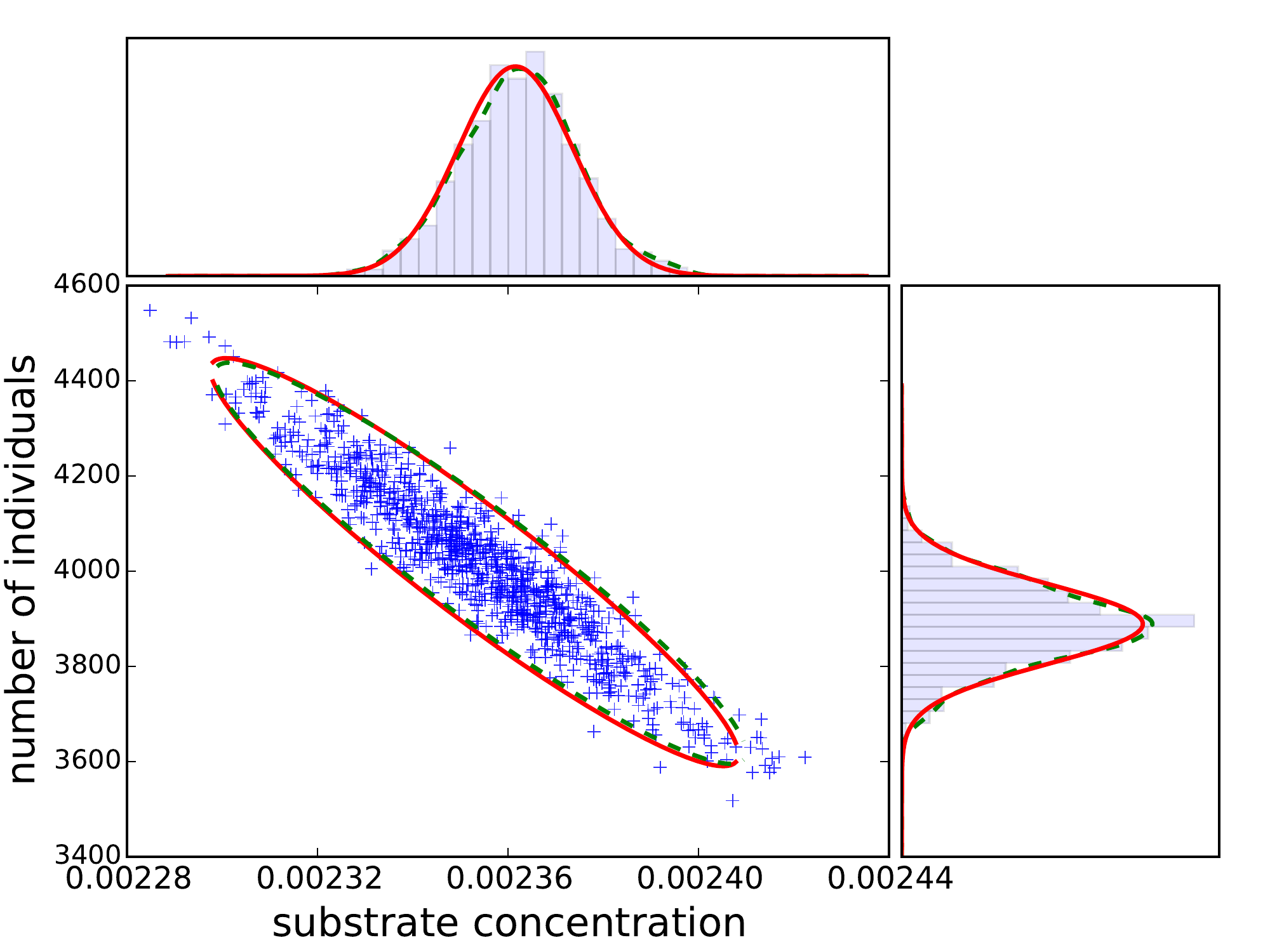}
\\ 
\scriptsize small population size : $V=10^{-8}$ l , $N_0 = 5$
&
\scriptsize medium population size : $V=10^{-6}$ l , $N_0 = 500$
\\
\scriptsize sample correlation : -0.947344
&
\scriptsize sample correlation : -0.942670
\end{tabular}
\includegraphics[width=7cm]{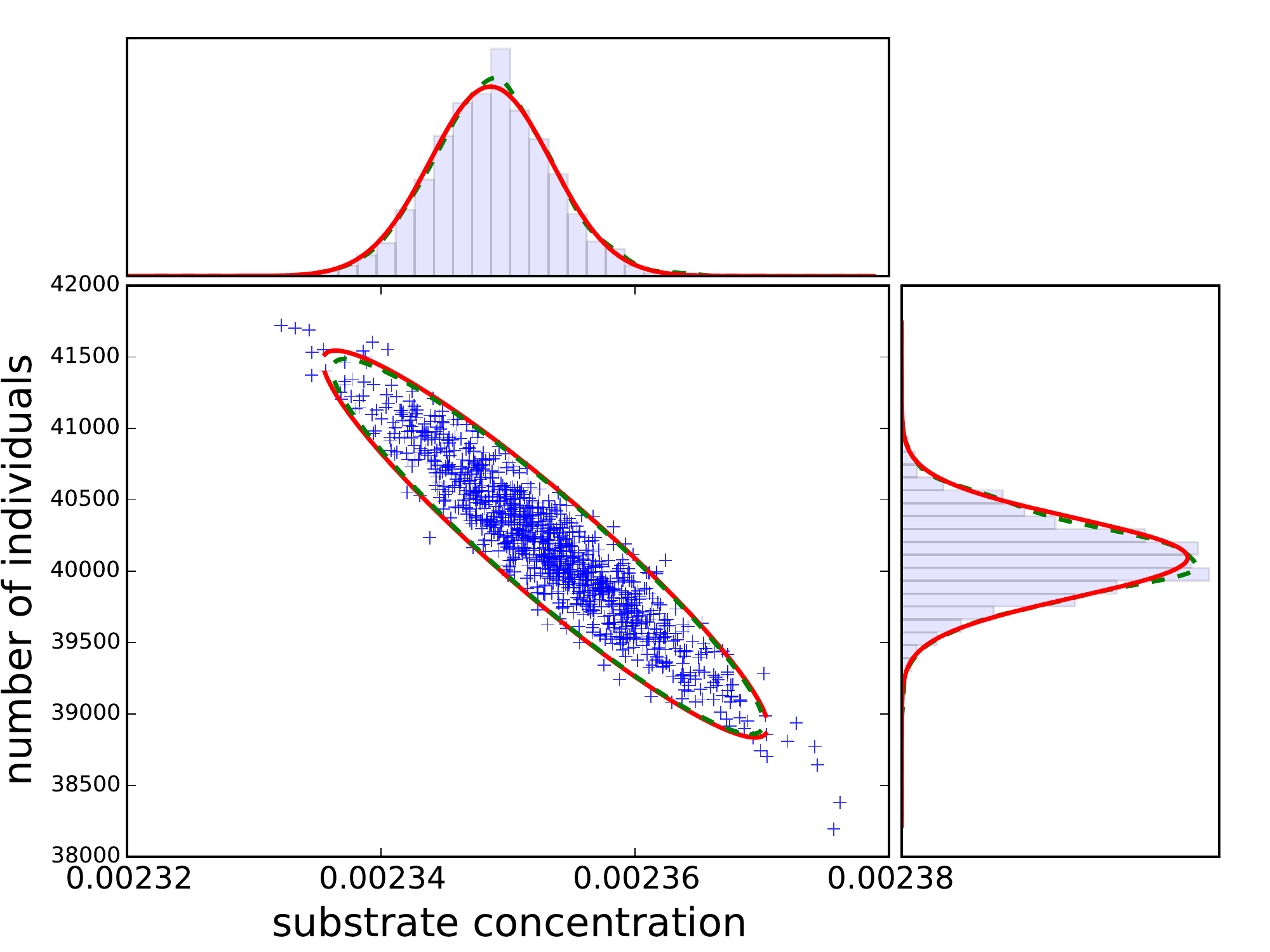}\\
\scriptsize large population size : $V=10^{-5}$ l , $N_0 = 5000$\\
\scriptsize \scriptsize sample correlation : -0.940845
\end{center}
\vskip-1em
\caption{\label{fig.QSD}
Distribution of the number of individuals and the substrate concentration, at time $T=1000$ h, for small (top left), medium (top right) and large (bottom) population sizes for the Monod growth model.
For each population size, the blue crosses represent the states of 1000 independent runs of the Crump-Young model. 
The green dashed ellipse is the 95\% confidence ellipse of a bidimensional normal variable where the mean and the covariance matrix are estimated on the 1000 Crump-Young model simulations (the sample correlation between the substrate concentration and the number of individuals is indicated under each graph). 
The red ellipse is the 95\% confidence ellipse of a normal variable with law \eqref{invariant.mesure2} (the theoretical correlation equals $-\frac{\mu(S^*)}{\mu'(S^*)\,\sqrt{\alpha\,\beta}}=-0.942470$).
On the top (resp. right) of each graph, the blue histogram represents the empirical distribution of the number of individuals (resp. the substrate concentration) of the Crump-Young model, the dashed green line is this distribution regularized with a Gaussian kernel and the red curve represents the probability density function of the Gaussian law $\mathcal{N}(N^*,\alpha)$ (resp. $\mathcal{N}(S^*, \beta)$), with $\alpha$ defined by \eqref{def.alpha} (resp. $\beta$ defined by \eqref{def.beta}), where $(N^*, S^*)$ is the non-trivial ($\neq (0,\Sin)$) equilibrium of the ODE system (see \eqref{invariant.mesure2}).
}
\end{figure*}

Figure \ref{fig.QSD}\footnote{The simulations of the 1000 runs in large population size was made on the \textit{babycluster} of the \textit{Institut \'Elie Cartan de Lorraine} : \url{http://babycluster.iecl.univ-lorraine.fr/}} compares the estimated quasi-stationary distribution of the Crump-Young model to the invariant mesure of the SDE given in Theorem \ref{th:intro-CY} for the three population sizes of Figure \ref{fig.CV.CY}. 
In small population size, we observe that the two laws are different. The main reason is the large probability of extinction of the Crump-Young model. Indeed, on the 1000 non-extinct populations, many are close to the extinction $(0,\Sin)$, whereas the invariant measure predicts a convergence in a neighbourhood of the non-trivial stationary state $(N^*,S^*)\neq (0,\Sin)$.
However, in medium and large population sizes, the invariant mesure \eqref{invariant.mesure2} is a very good approximation of the quasi-stationary distribution of the Crump-Young model.

\subsubsection{Haldane growth}
We now use the following Haldane growth function:
$$
	\mu(S)=1.35\,\frac{S}{0.004+S+S^2/0.04}\,,
$$
and $k = 0.23$, $m = 7\,. 10^{-13}\, g$, $D=0.5$ h$^{-1}$, 
$\Sin=0.0978$ g.l$^{-1}$.

The behavior of the chemostat, for Haldane growth, depends on the initial condition. Indeed, there is, for the ODE, two basins of attraction which are associated to the two stable equilibria $(0,\Sin)$ and $(N^*,S^*)$ (see Figure \ref{fig.haldane}), contrary to Monod growth for which there is only one stable equilibrium (the washout is an unstable equilibrium).

If the initial condition is close to the boundary of the two basins of attraction, the ODE remains in its initial basin and converges to its attractor whereas, due to the randomness, the Crump-Young model can change basin of attraction. 
The SDE \eqref{SDE} is very depending on the ODE solution and will converge to the invariant mesure of the basin of attraction associated to the initial condition. Therefore, the SDE \eqref{SDE} is not representative of the two possible convergences for one given initial condition.
The SDE \eqref{SDE} is in fact a good approximation of the Crump-Young model when the population size is sufficiently large (which depends on the distance between the initial condition and the boundary of the two basins of attraction) to ensure that the Crump-Young model does not change (with a large probability) basin of attraction (see Figure \ref{fig.haldane.large.pop}).
Even if the approximation only holds for large population, in Figure \ref{fig.haldane.large.pop} (right), both models converge to the population extinction (even if the SDE is not absorbed, it converges to $0$).

However, if the object of interest is the convergence towards $(0,\Sin)$ or $(N^*,S^*)$ for one given initial condition (close to the boundary of basins of attraction) then we must either use the Crump-Young model (if the simulation time is reasonable) or use a model 
which keeps more qualitative properties than the SDE \eqref{SDE}.

\begin{figure*}
\begin{center}
\begin{tabular}{cc}
\includegraphics[width=6cm]{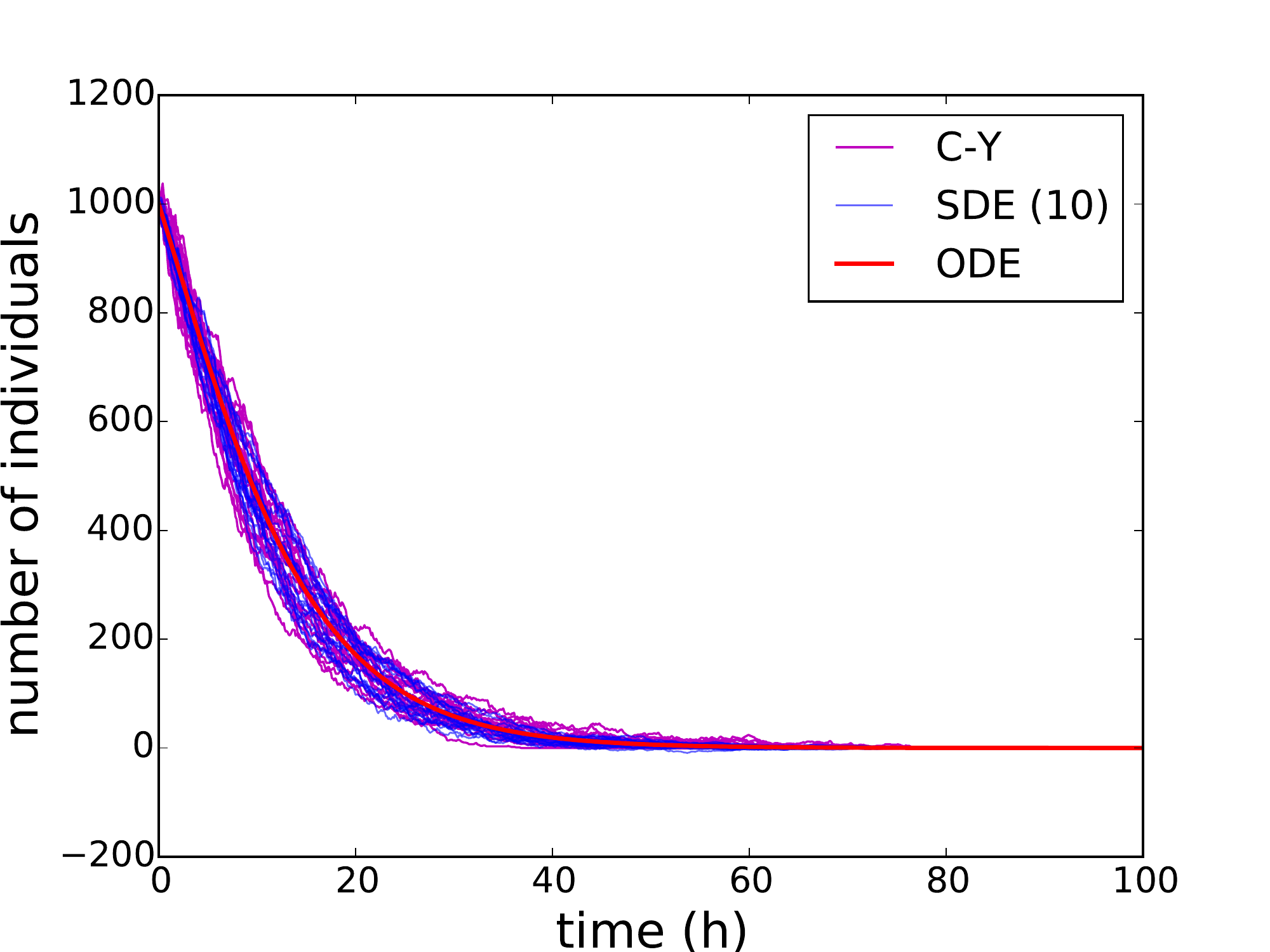}
&
\includegraphics[width=6cm]{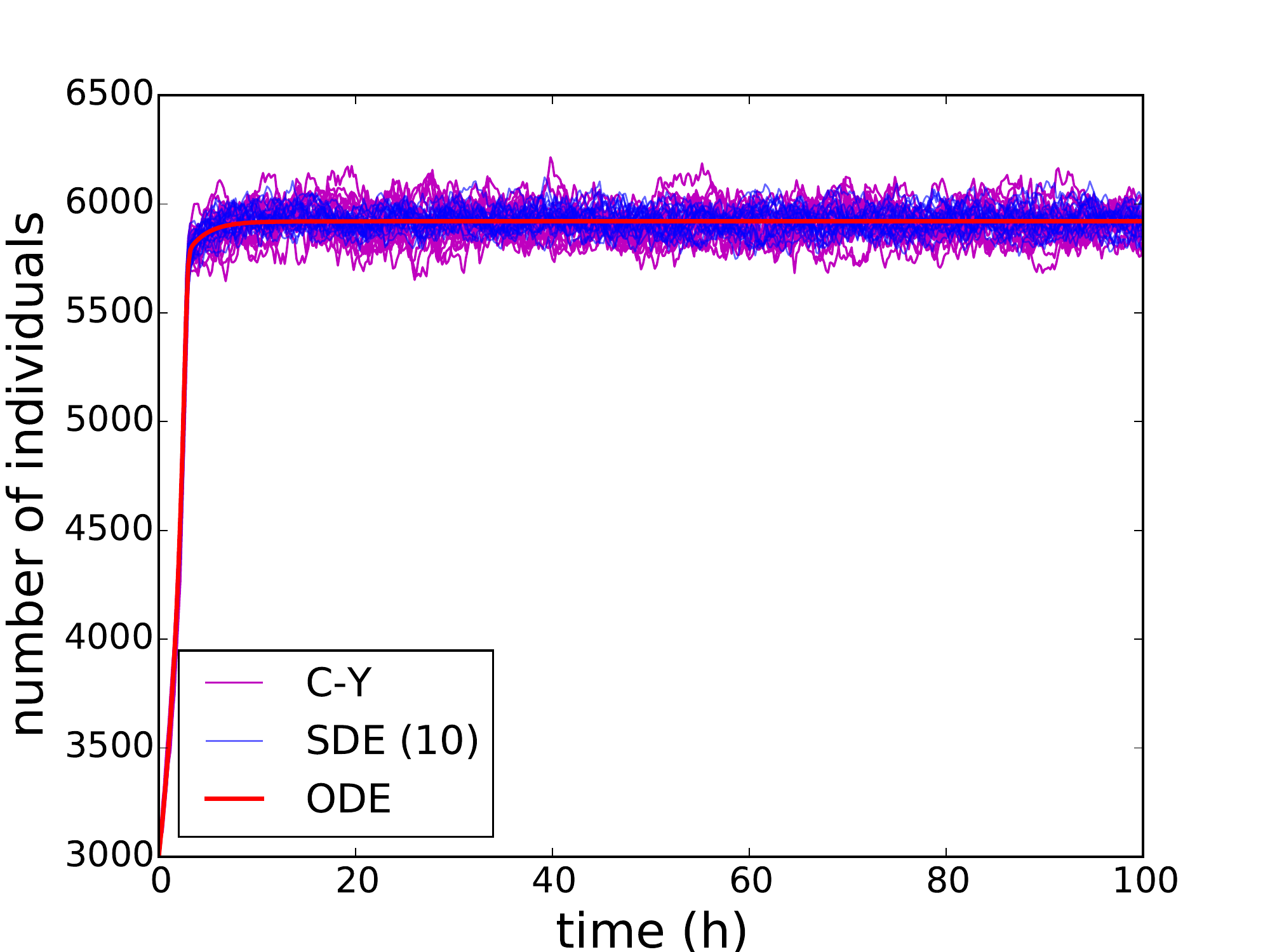}
\\ 
\scriptsize $N_0 = 1000$, $S_0=0.08$ g.l$^{-1}$
&
\scriptsize $N_0 = 3000$, $S_0=0.04$ g.l$^{-1}$
\\ 
\scriptsize $V=10^{-8}$ l
&
\scriptsize $V=10^{-8}$ l
\end{tabular}
\end{center}
\vskip-1em
\caption{Time evolution of the number of individuals for the Haldane growth model for 20 independent runs of the Crump-Young model (magenta lines), 20 independent runs of the SDE \eqref{SDE} (blue lines) and the ODE (red curve) for initial conditions $N_0 = 1000$, $S_0=0.08$ g.l$^{-1}$ (left) and $N_0 = 3000$, $S_0= 0.04$ g.l$^{-1}$ (right).
\label{fig.haldane.large.pop}
}
\end{figure*}

In fact, Theorem \ref{th:tcl-CY} suggests that, for $n$ sufficiently large, the Crump-Young model can be approximated by
$
(N_t^n,\, S^n_t) \approx (\widetilde N_t^n ,\, \widetilde S_t^n)
$ with
\begin{align*}
\widetilde N_t^n := n\, N_t + \sqrt{n}\, \left[Q_t + F^n_t\right], \quad 
\widetilde S_t^n := S_t + \frac{1}{\sqrt{n}}\, \left[R_t + H^n_t\right]\,,
\end{align*}
where $(F^n)_n$ and $(H^n)_n$ are two sequences of processes which converge in distribution towards the process 0 in  $\DD([0,T],\RR)$. The SDE \eqref{SDE} is obtained by letting $F^n=H^n=0$. Let now consider $(F^n)_n$ and $(H^n)_n$ be defined by
\begin{align*}
\dif F_t^n 
	& = 
		\left[\sqrt{n}\,\left(\mu(S_t^n)-\mu(S_t)\right)\,\frac{N_t^n}{n}
						-\mu'(S_t)\,R_t\,N_t \right]\,\dif t
\\
	& \quad
				+ \left[\sqrt{(\mu(S_t^n)+D)\,\frac{N_t^n}{n}}
						-\sqrt{(\mu(S_t)+D)\,N_t}\right]\,\dif B_t
\end{align*}
and
$$
	\dif H_t^n = \frac{k}{V}\,m\,\left[
		\mu(S_t)\,(Q_t-Q_t^n)+R_t\,\mu'(S_t)\,N_t-\sqrt{n}\,(\mu(S_t^n)-\mu(S_t))\frac{N_t^n}{n}
	\right]\,\dif t\,,
$$
with initial condition $F_0^n=H_0^n=0$.
Following, for example, the approach used in the proof of Lemma \ref{prop.identification.Rt}, we can prove that $(F^n)_n$ and $(H^n)_n$ converge towards 0 in distribution. We then (heuristically) obtain the following model of approximation : 
\begin{equation} 
\label{eq:EDS_new}
     \begin{cases}
\dif \widetilde N^n_t &= (\mu(\widetilde S^n_t)-D)\,\widetilde N^n_t\,\dif t 
					+ \sqrt{(\mu(\widetilde S^n_t)+D)\,\widetilde N^n_t}\,\dif B_t\,, \\
\dif \widetilde S^n_t &= \left[D\,(\Sin-\widetilde S^n_t)
					-\frac{k}{V\,n}\,m\,\mu(\widetilde S^n_t)\, \widetilde N^n_t \right] \dif t\,.
     \end{cases}
\end{equation}

This new approximation model can be seen as a particular case of the model of \cite{campillo2011chemostat}; see \cite[Equations (17a) and (17b)]{campillo2011chemostat} with $K_1=K_4=1$ and $K_2=K_3=K_5=+\infty$ (note that $K_2=K_3=K_5=+\infty$ corresponds to a continuous approximation of the substrate equation for a large number of substrate particles which is an approximation that we made for all models in this article).

In contrast with \eqref{SDE}, the variance of the population size $\widetilde{N}^n_t$  depends on itself. Moreover this type of dependence  is classic in population dynamics, see for instance \cite{BM15}.

\begin{figure*}
\begin{center}
\begin{tabular}{c@{\hskip-0.8em}c@{\hskip-0.8em}c}
\includegraphics[width=5.5cm]{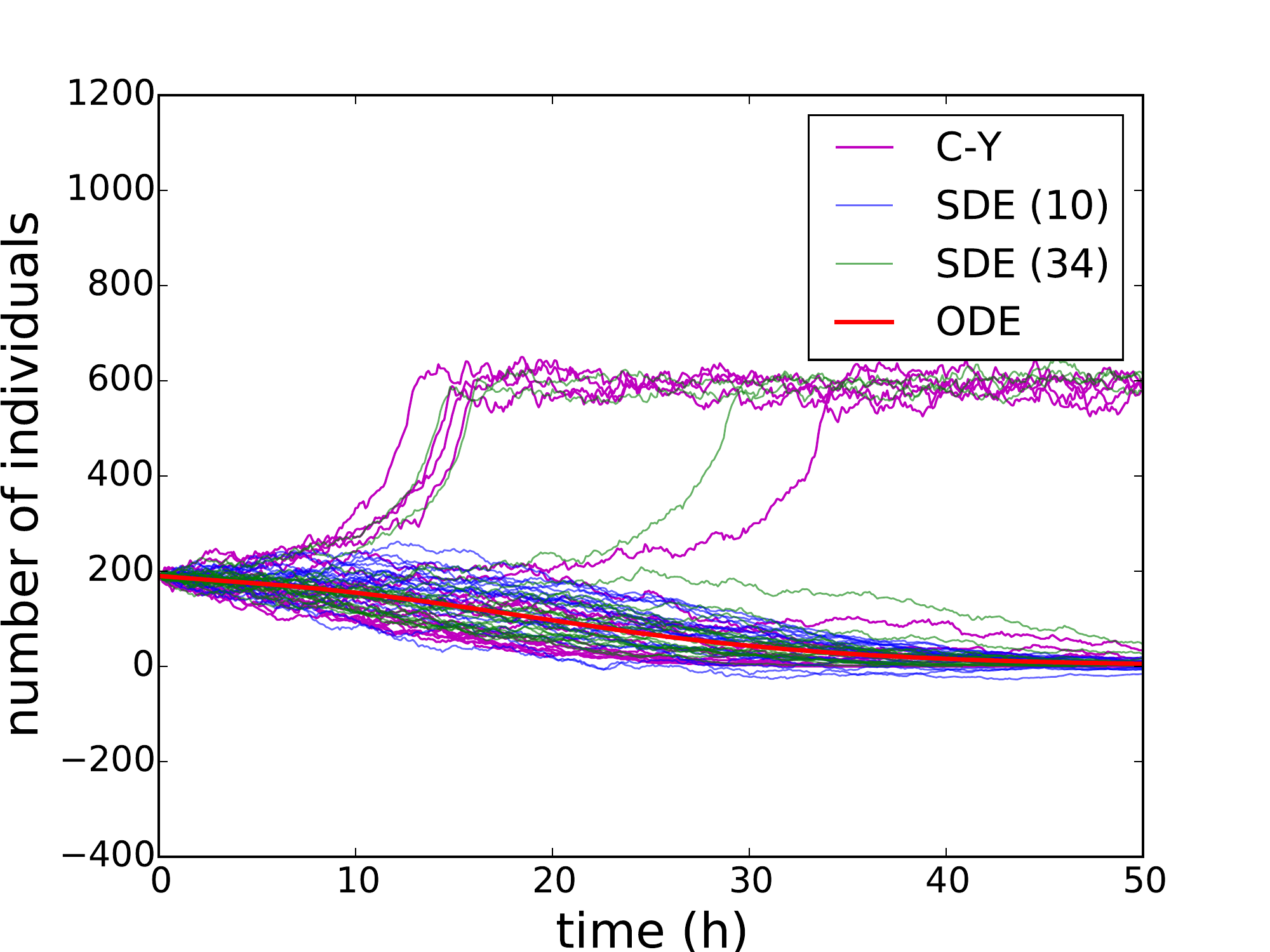}
&
\includegraphics[width=5.5cm]{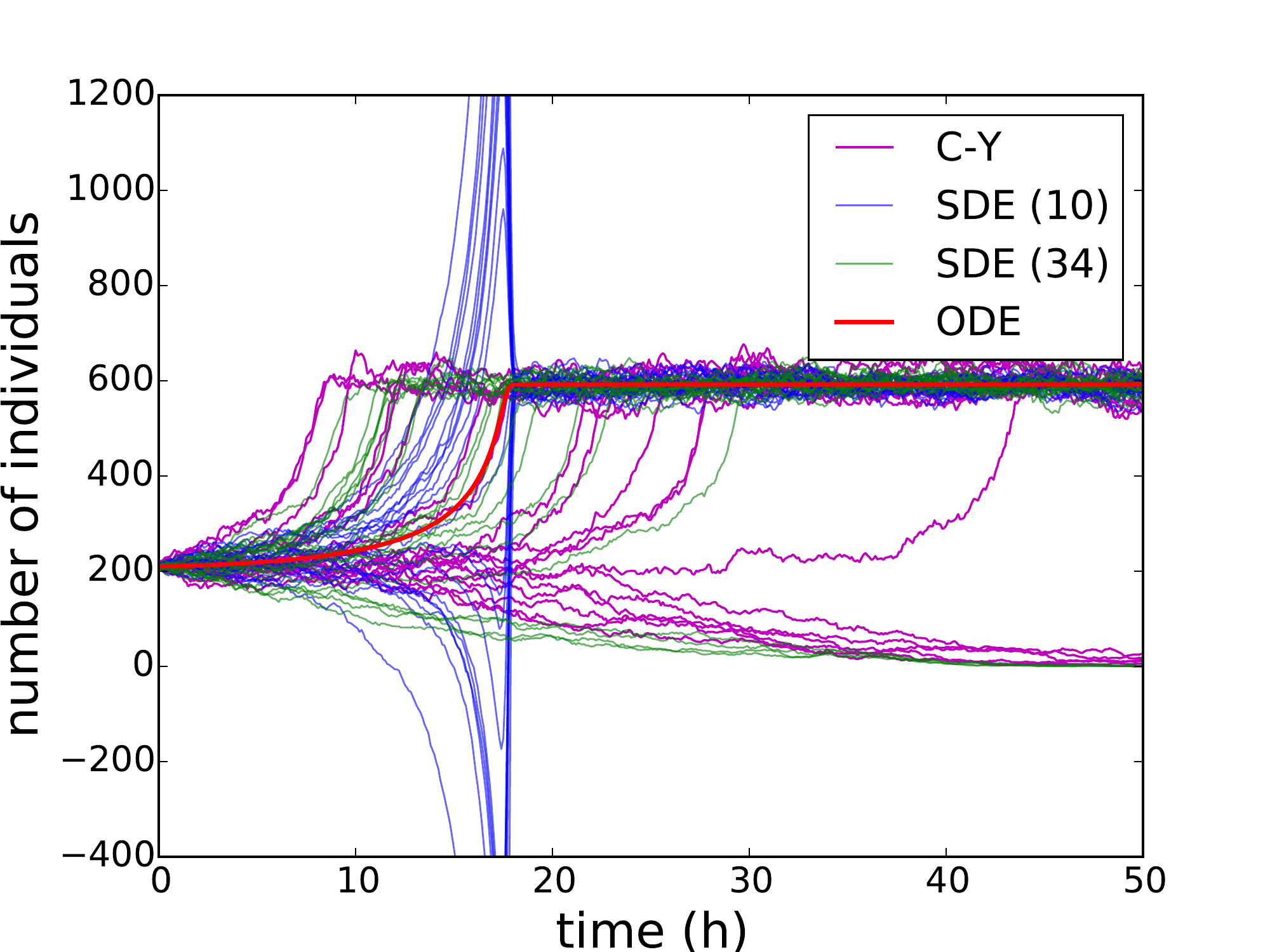}
&
\includegraphics[width=5.5cm]{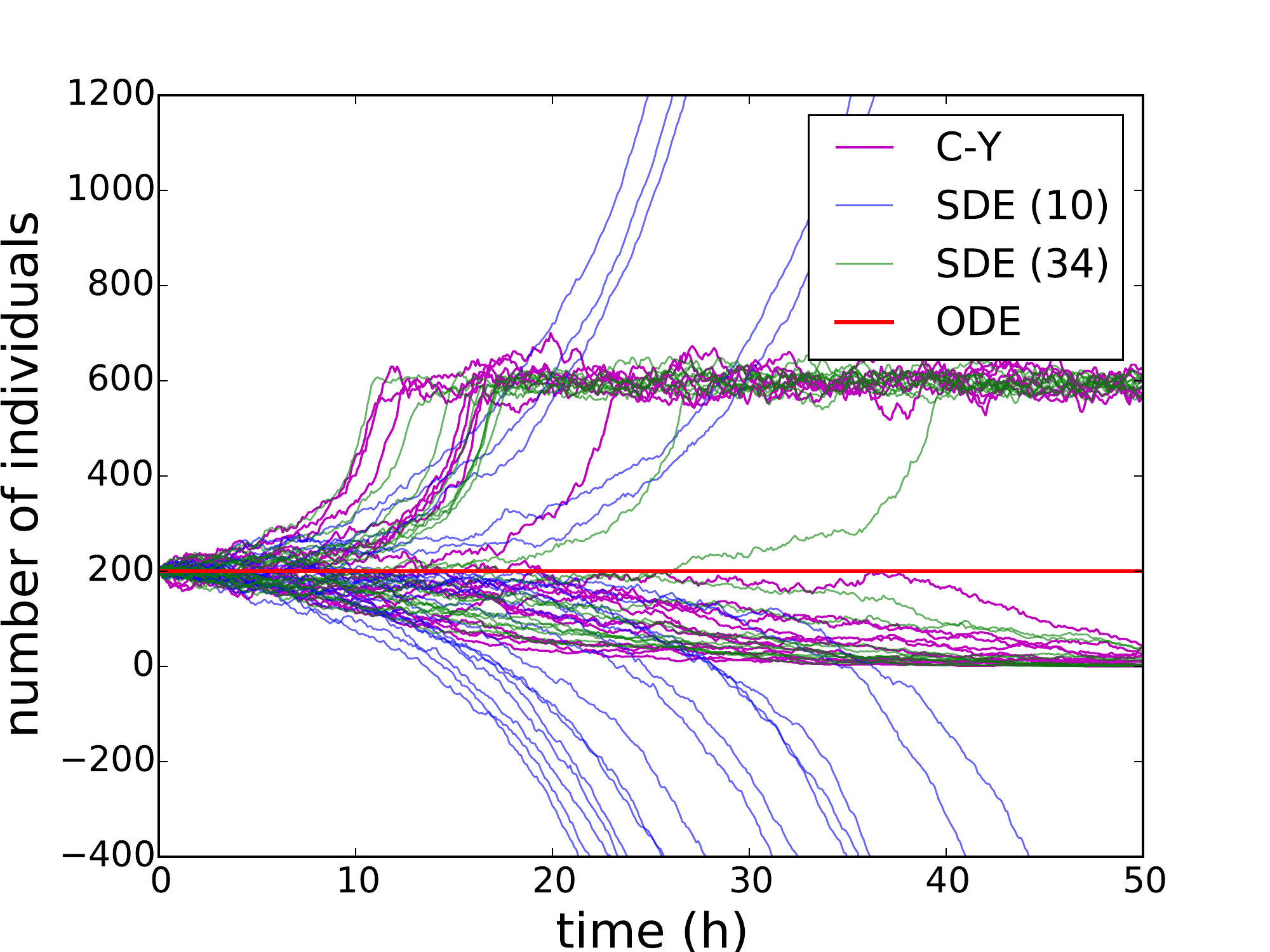}
\\ 
\includegraphics[width=5.5cm]{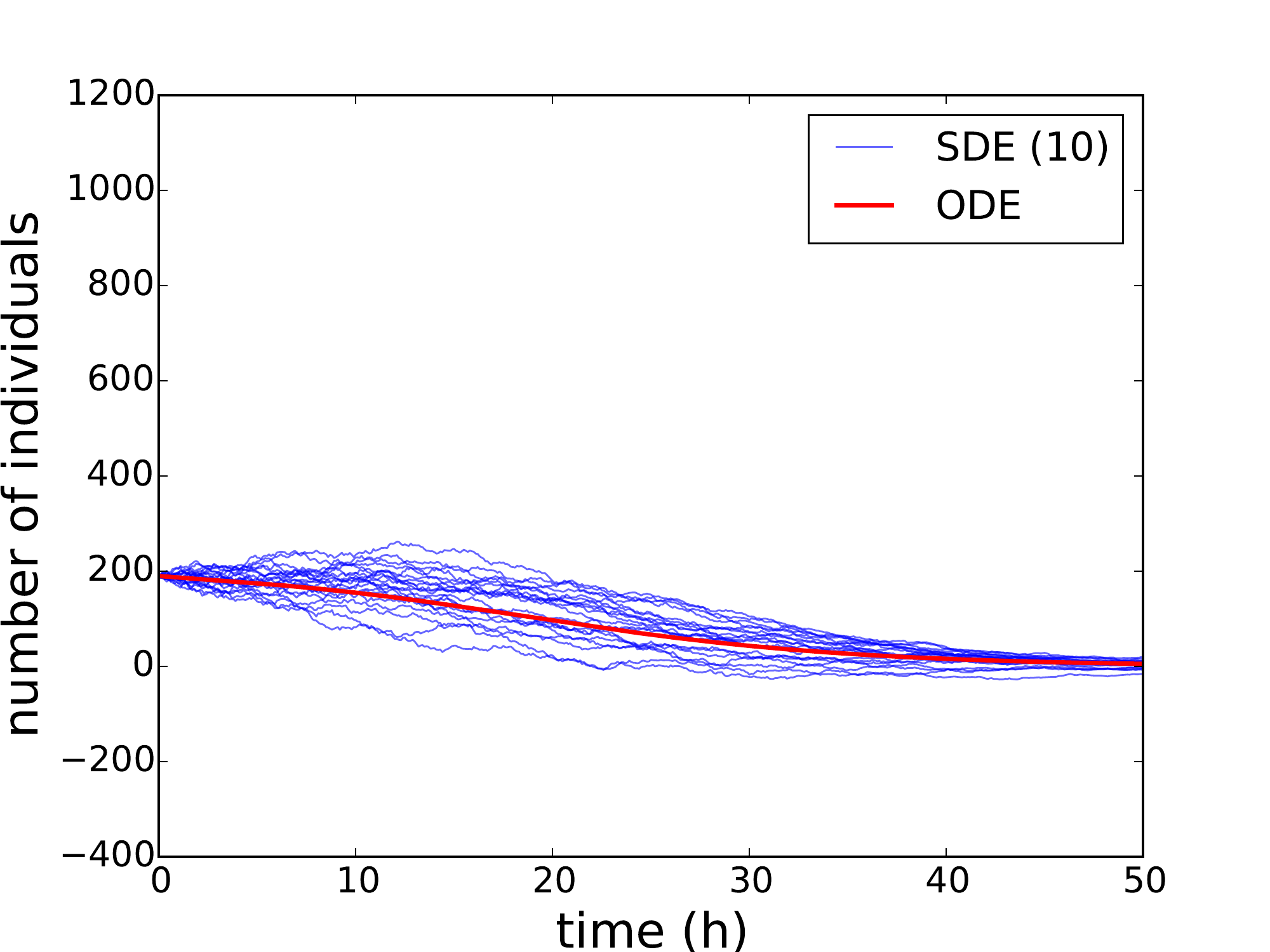}
&
\includegraphics[width=5.5cm]{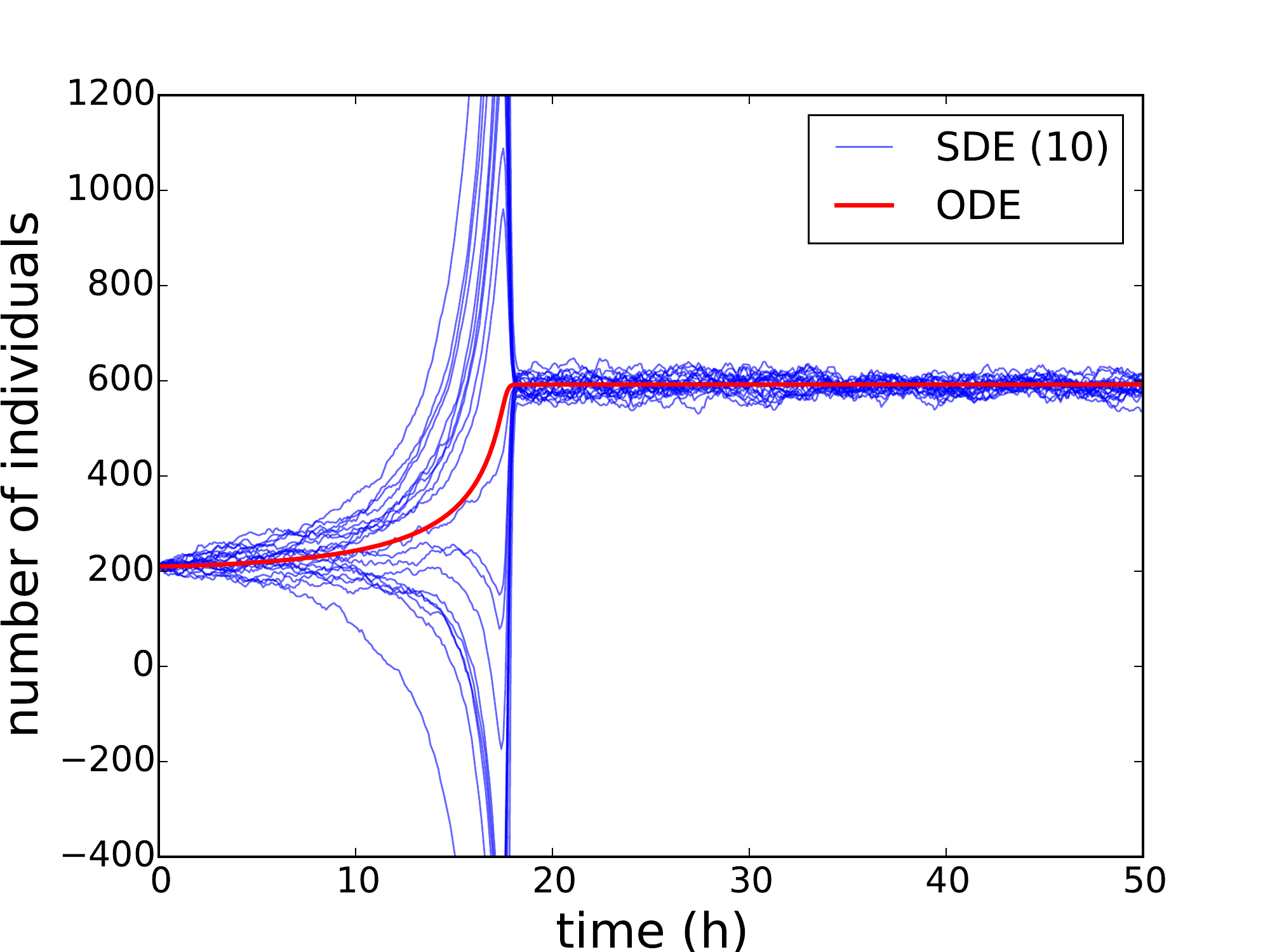}
&
\includegraphics[width=5.5cm]{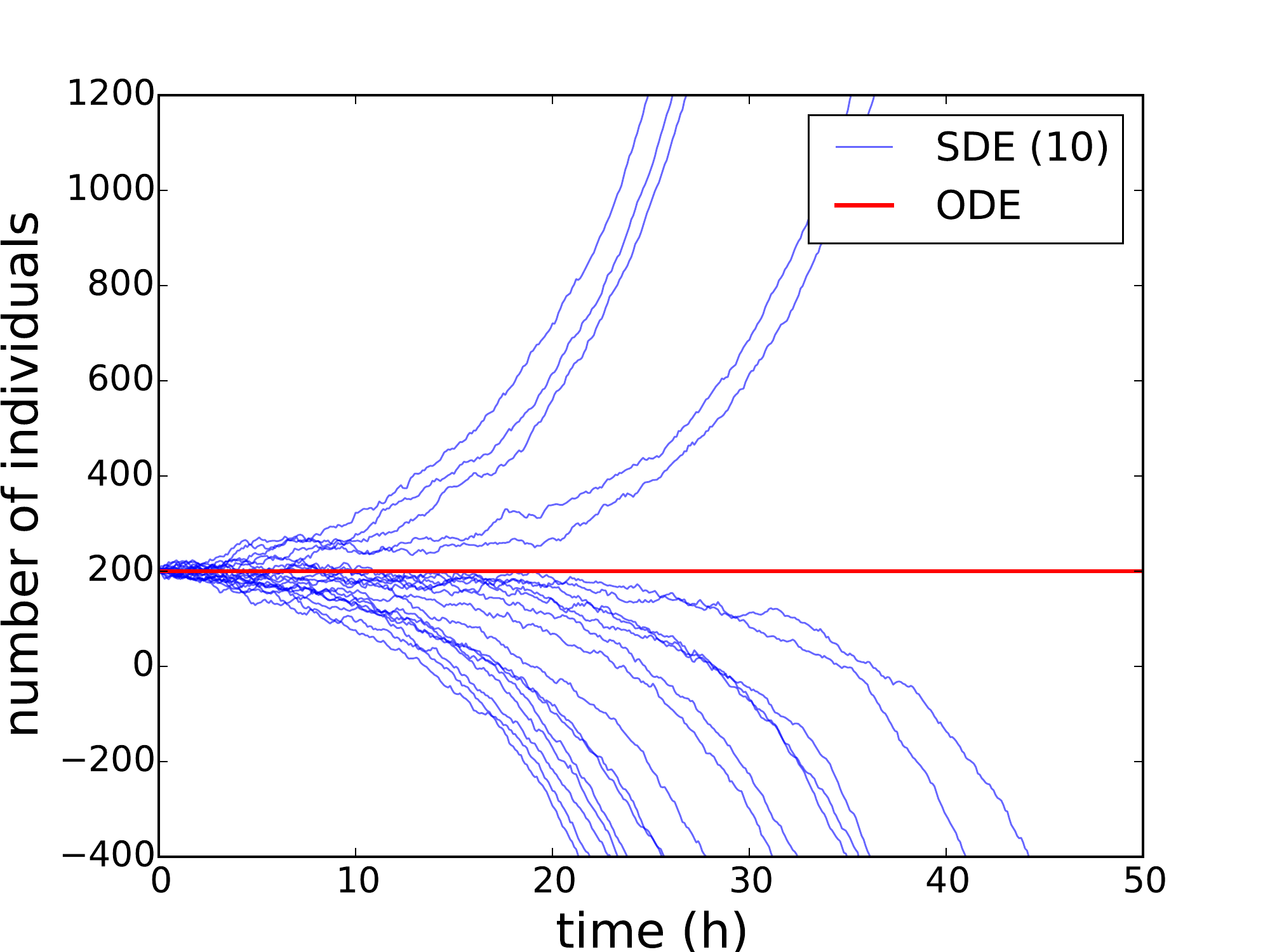}
\\ 
\includegraphics[width=5.5cm]{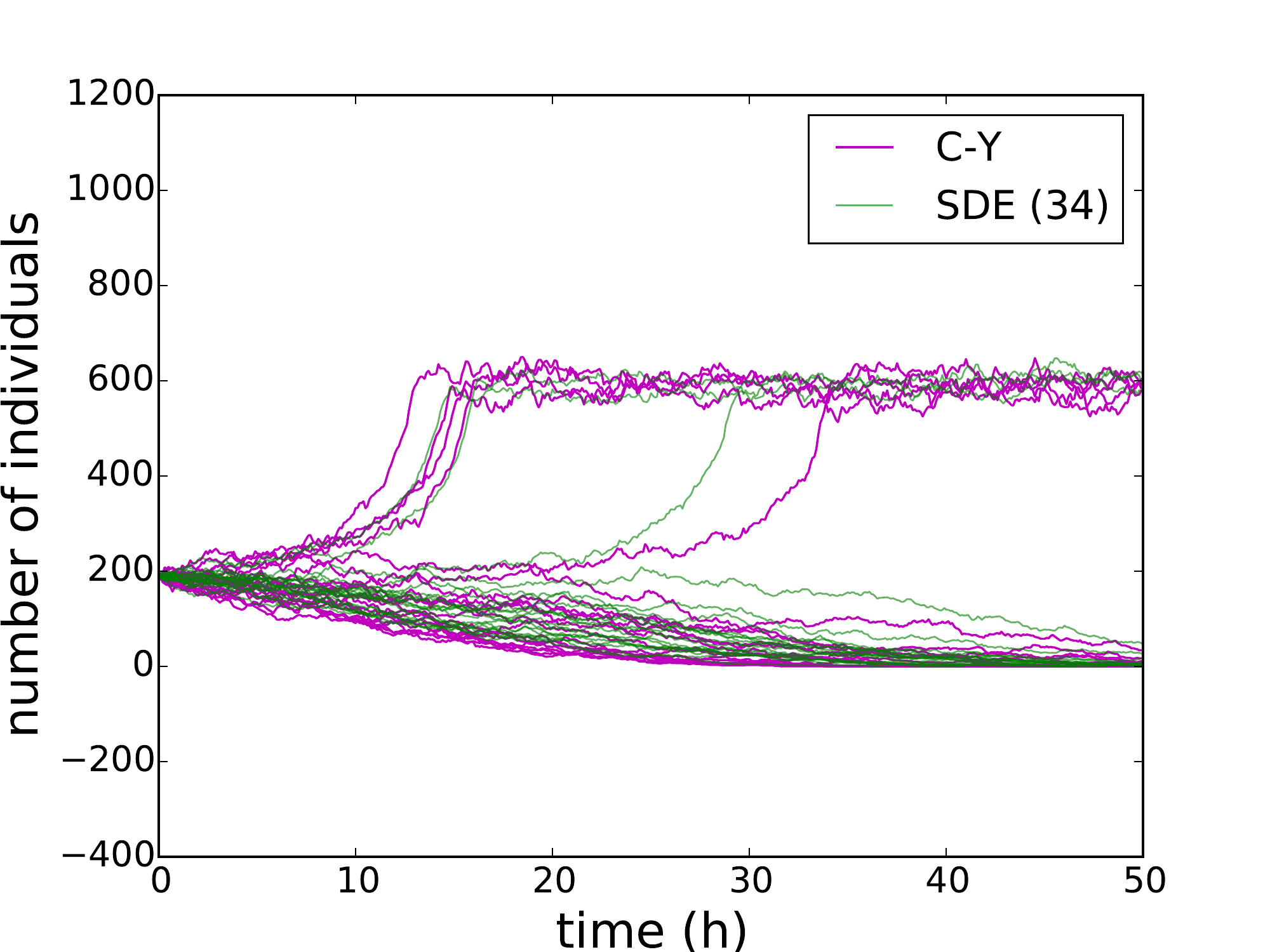}
&
\includegraphics[width=5.5cm]{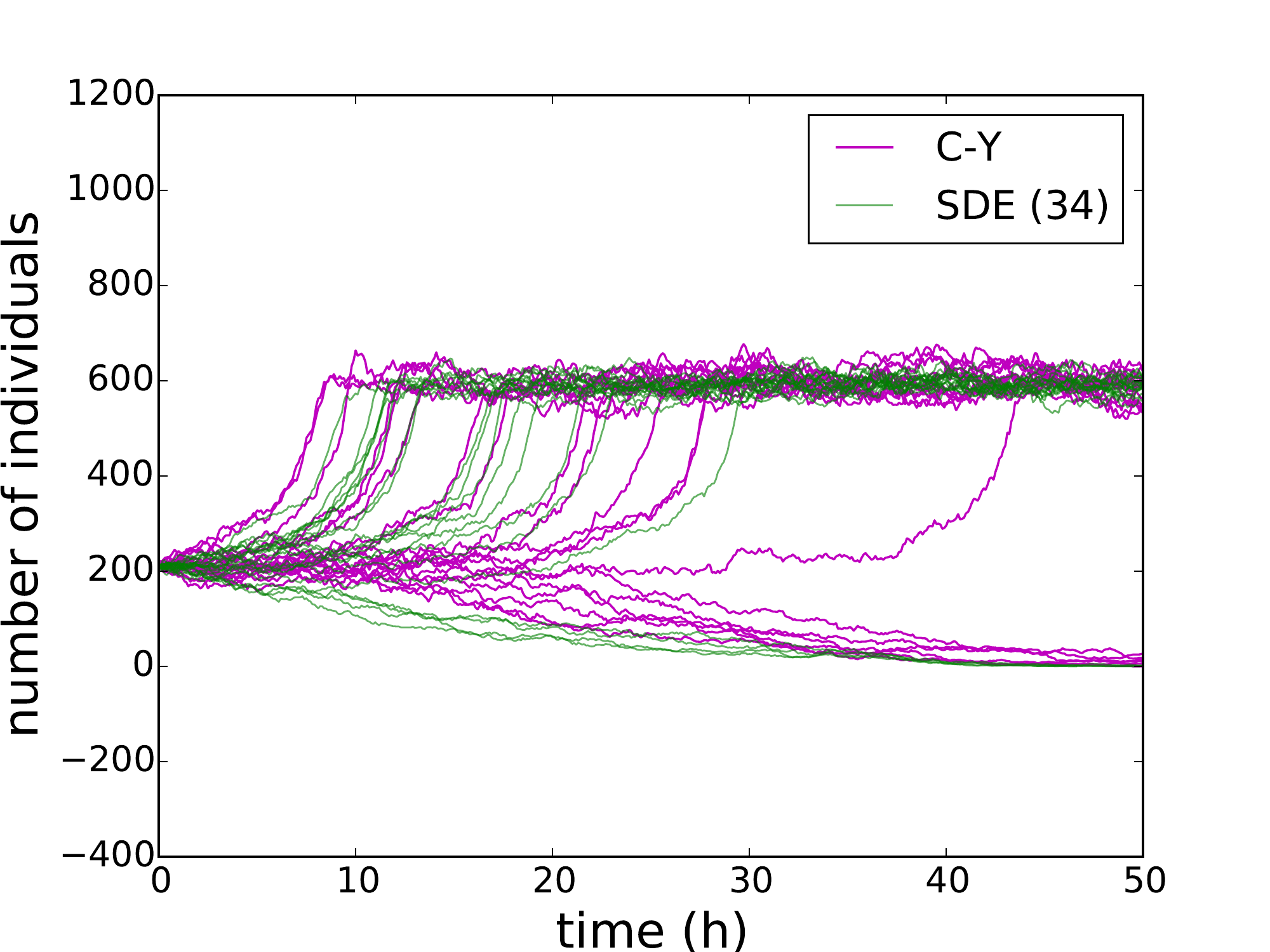}
&
\includegraphics[width=5.5cm]{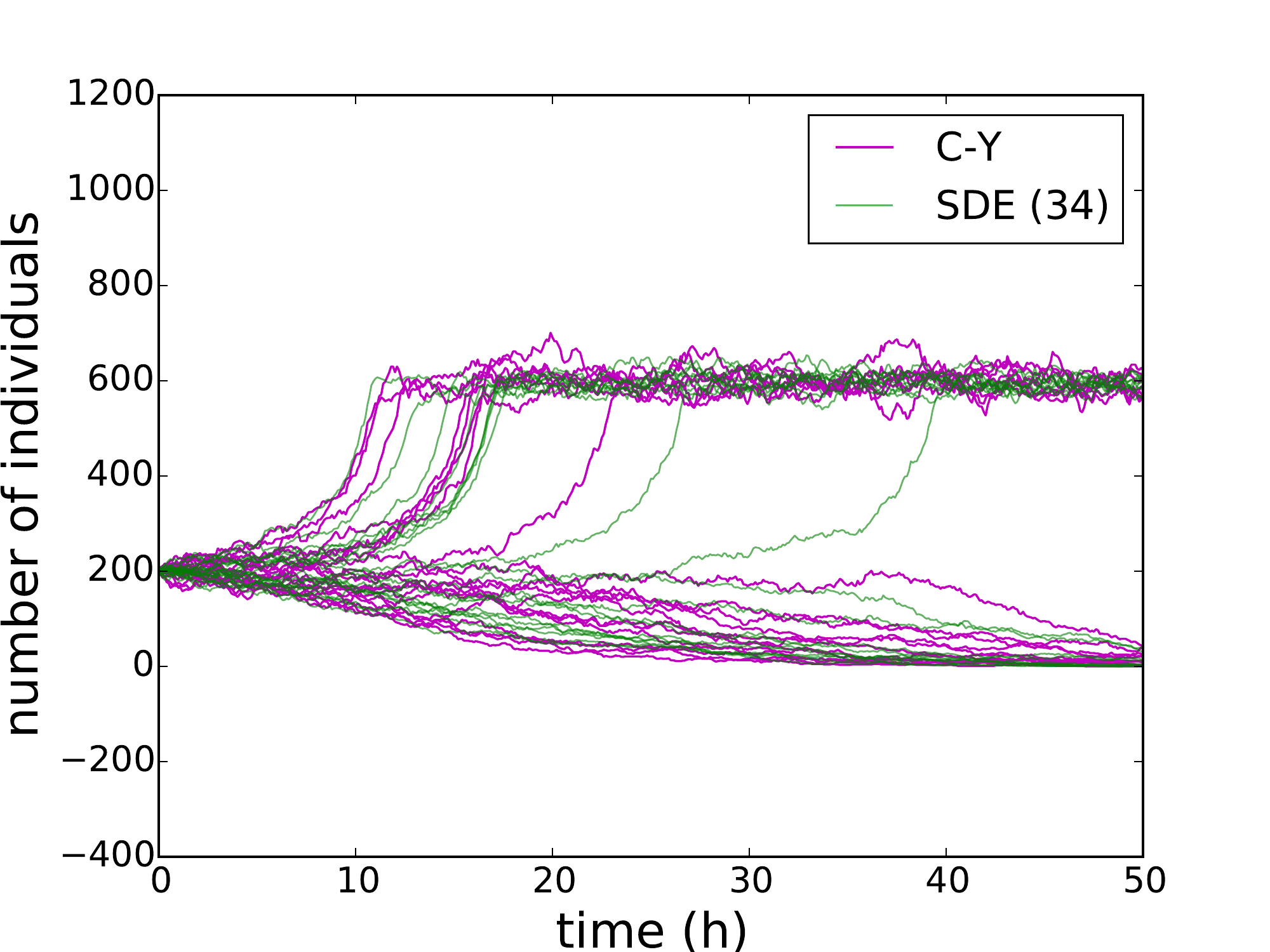}
\\ 
\scriptsize $N_0 = 190$, $S_0=0.07$ g.l$^{-1}$
&
\scriptsize $N_0 = 210$, $S_0=0.065$ g.l$^{-1}$
&
\scriptsize $N_0 = N^{ue} = 200$, $S_0= S^{ue} = 0.0656$ g.l$^{-1}$
\\ 
\scriptsize $V=10^{-9}$ l
&
\scriptsize $V=10^{-9}$ l
&
\scriptsize $V=10^{-9}$ l
\end{tabular}
\end{center}
\vskip-1em
\caption{Time evolution of the number of individuals for the Haldane growth model for 20 independent runs of the Crump-Young model (magenta lines), 20 independent runs of the SDE \eqref{eq:EDS_new} (green lines), 20 independent runs of the SDE \eqref{SDE} (blue lines) and the ODE (red curve) for initial conditions $N_0 = 190$, $S_0=0.07$ g.l$^{-1}$ (left), $N_0 = 210$, $S_0=0.065$ g.l$^{-1}$ (middle) and $N_0 = N^{ue}$, $S_0= S^{ue}$ g.l$^{-1}$ (right).
\label{fig.new.SDE}
}
\end{figure*}

\medskip

Figure \ref{fig.new.SDE} represents the time evolution of the number of individuals for the four models (ODE, Crump-young models, SDE \eqref{SDE} and SDE \eqref{eq:EDS_new}) in three cases. Each column represents the same simulation with all or some represented curves (the first line allows to compare the four models together, however we have split each graph in two graphs for the sake of clarity). The first one (on the left) is for initial condition close to the unstable equilibrium $(N^{ue},S^{ue})$ for which the solution of the ODE converges towards the washout $(0,\Sin)$. Therefore the solutions of the SDE \eqref{SDE} also converge towards the washout. However the Crump-Young model changes basin of attraction with a large probability and converges either to the washout or to a neighbourhood of the stable equilibrium $(N^*,S^*)$. As we can observe, the solutions of the SDE \eqref{eq:EDS_new} mimic the behavior of the Crump-Young model and then SDE \eqref{eq:EDS_new} seems better than the SDE \eqref{SDE} in this context. 

The second case (center) is for initial condition close to the unstable equilibrium $(N^{ue},S^{ue})$ for which the solution of the ODE converges towards the stable equilibrium $(N^*,S^*)$. Once again, the Crump-Young model and the diffusion process \eqref{eq:EDS_new} depict two possible convergences (towards the washout or the quasi-stationary distribution in the area of $(N^*,S^*)$) while the diffusion process \eqref{SDE} follows the solution of the ODE. We see an explosion of the noise for the blue curve. This comes from that for small time $t$, the matrix  $A_t$ has large positive eigenvalue (due to the initial condition, recall that it tends to infinity when the initial condition is the unstable equilibrium) but, as in the Monod case, $\Sigma_t$ converges to a finite matrix.

For the last case (right), the initial condition equals the unstable equilibrium $(N^{ue},S^{ue})$. Therefore, the deterministic approximation stays at this equilibrium whereas the Crump-Young model and the diffusion process \eqref{eq:EDS_new} depict over again the two possible convergences.
We observe that the solutions of \eqref{SDE} diverge. In fact, we can write the SDE as in \eqref{eq.exp.Z}, but, as $\mu'(S^{ue})<0$, the eigenvalue $\lambda_s^1=-\frac{k}{V}\,m\,\mu'(S^{ue})\,N^{ue}$ of $A_s$ is positive which implies the divergence of the SDE \eqref{SDE}.

Even if the solutions of \eqref{eq:EDS_new} seems, for some parameters, to be a more suitable approximation for the Crump-Young model, it is nevertheless more difficult to study it mathematically. Indeed, as for the Crump-Young model, there is always extinction; see Theorem \ref{lem:eds-ext}. Also, even if this process is continuous and solution to a stochastic differential equation, it is not possible to deduce a result of uniqueness (or convergence) for a quasi-stationary distribution because it is not reversible in contrast with the classic logistic diffusion process; see \cite{CMS13}. Also, in contrast with the solutions of \eqref{SDE}, no explicit formula is known for solutions of \eqref{eq:EDS_new}.

\subsection{Extinction time of the diffusion process \eqref{eq:EDS_new}}
\label{subsec.extin.newSDE}

In this section, we will consider a solution $(\widetilde{N}_t,\widetilde{S}_t)_{t \geq 0}$ of \eqref{eq:EDS_new} for one fixed $n$. The notation $\mathbb{P}_{(x,s)}$ refers to the probability given the initial condition is $(\widetilde{N}_0,\widetilde{S}_0)=(x,s)$ and $\mathbb{E}_{(x,s)}$ is the expectation associated to this probability.

\begin{theorem}[Extinction]
\label{lem:eds-ext}
Let $(\widetilde{N}_t,\widetilde{S}_t)_{t \geq 0}$ be a solution of \eqref{eq:EDS_new} and 
$$
T_0= \inf\{t \geq 0 \ | \ \widetilde{N}_t=0 \}.
$$
Then $\mathbb{P}_{(x,s)}(T_0 < + \infty)=1$ for any starting distribution $(x,s)\in \mathbb{R}_+^2$. Moreover, there exist $C,\alpha>0$ such that for all $(x,s)\in \mathbb{R}_+^2$ and $t\geq 0$,
\begin{equation}
\label{eq:proba-ext}
\mathbb{P}_{(x,s)}(T_0>t) \leq C e^{-\alpha t} (x+s+1).
\end{equation}

\end{theorem}

\begin{proof}

First, we assume that for every compact set $K\subset\mathbb{R}^2_+$, there exist $t_0,\delta>0$ such that
\begin{equation}
\label{eq:small}
\delta=\inf_{(x,s)\in K} \mathbb{P}_{(x,s)}(T_0 < t_0) >0.
\end{equation}
Secondly, considering $V_0: (N,S) \mapsto \frac{km}{V\,n} N + S-\Sin$ and using \eqref{eq:EDS_new}, $(e^{D t} V_0(\widetilde N_t, \widetilde S_t) )_{t\geq 0}$ is a martingale; namely $V_0$ is a Lyapunov-type function.
From \eqref{eq:small} and the Lyapunov property, it is then classic to prove the statement of the Lemma.
Indeed, shortly, the Lyapunov property entails that, whatever the initial position is, the process converges rapidly in a compact set and then, by \eqref{eq:small}, it will be absorbed in finite time.
This standard argument to prove geometric ergodicity of general Markov processes is given, for instance, by \cite{HM11}. Nevertheless we can not directly apply this theorem because even if $\tilde{S}_t \to \Sin$, it does not reach it, therefore we can not obtain the convergence of $\tilde{S}_t$ to $\Sin$ in total variation.

So let us prove that the Lyapunov property and \eqref{eq:small} are sufficient to ensure \eqref{eq:proba-ext}. Let us fix a compact set $K\subset \RR^2_+$ such that, for all $x\notin K$, $V_0(x)\geq C_0$, for some $C_0>0$. Moreover let us fix the associated $t_0$ and $\delta$ as in \eqref{eq:small}.
 
We divide the proof in two steps.
\subsubsection*{Step 1 : Bound on hitting time}
Let $\tau$ be the hitting time of $K$. Using the stopping-time theorem, for any $n\in \mathbb{N}$ and $(x,s) \notin K$, we have
$$
\mathbb{E}_{(x,s)}[e^{D \tau \wedge n}V_0(\widetilde N_{\tau \wedge n}, \widetilde S_{ \tau \wedge n})] = V_0(x,s).
$$
Then
$$
C_0  \mathbb{E}_{(x,s)}[e^{D (\tau \wedge n)}] \leq V_0(x,s).
$$
Indeed, if $(\widetilde N_{0}, \widetilde S_{0}) \notin K$ then $(\widetilde N_{t}, \widetilde S_{t}) \notin K$ for all $t\leq \tau$ (by definition of $\tau$).
Using the monotone convergence theorem, we have for every $(x,s) \notin K$
\begin{equation*}
\mathbb{E}_{(x,s)}[e^{D \tau}] \leq C_0^{-1}\, V_0(x,s)\,.
\end{equation*}
Moreover, if $(x,s) \in K$, then $\tau=0$ hence
$\mathbb{E}_{(x,s)}[e^{D \tau}] \leq 1$ therefore for any $(x,s)\in \RR_+^2$
\begin{equation}
\label{eq:momentexpo}
\mathbb{E}_{(x,s)}[e^{D \tau}] \leq C_0^{-1}\, V_0(x,s)+B
\end{equation}
with $B=C_0^{-1}\Sin+1$ (because $V_0(x,s) + \Sin \geq 0$).
Then the Markov inequality gives
\begin{equation}
\label{eq:chernoff}
\mathbb{P}_{(x,s)}(\tau \geq t) = \mathbb{P}_{(x,s)}(e^{D \tau} \geq e^{D t}) \leq e^{-D t} \mathbb{E}_{(x,s)}[e^{D \tau}]\leq (C_0^{-1}\, V_0(x,s)+B) \, e^{- D t}\, .
\end{equation}

\subsubsection*{Step 2 : Bound on the extinction time}

\newcommand{\indice}{\ell}

Let  $s_0=0$ and for every $\indice\geq 0$,
$$
\tau_{\indice+1} = \inf\{s\geq s_\indice \ | (\widetilde N_s, \widetilde S_s) \in K  \} - s_\indice
$$
$$
s_{\indice+1} =s_\indice + \tau_{\indice+1} + t_0.
$$
Let $\theta \in [0,1]$, by H\"{o}lder inequality, we have
\begin{align*}
\mathbb{P}_{(x,s)}(T_0>t)
&=\mathbb{P}_{(x,s)}(\widetilde N_{t} \neq 0)
=\sum_{\indice \geq 0} \mathbb{P}_{(x,s)}(  \widetilde N_{t} \neq 0, \ t \in [s_\indice,s_{\indice+1}) )\\
&\leq\sum_{\indice \geq 0} \mathbb{P}_{(x,s)}(  \widetilde N_{s_\indice} \neq 0, \ t \in [s_\indice,s_{\indice+1}) )\\
&\leq \sum_{\indice \geq 0} \mathbb{P}_{(x,s)}(  \widetilde N_{s_\indice} \neq 0)^{\theta}  \mathbb{P}_{(x,s)}(t \in [s_\indice,s_{\indice+1}) )^{1-\theta}.
\end{align*}
On the first hand and if $\indice\geq 1$, by the strong Markov property, Equation \eqref{eq:small} and an induction argument, we have
\begin{align*}
\mathbb{P}_{(x,s)}(  \widetilde N_{s_\indice} \neq 0)
&=\mathbb{E}_{(x,s)}\left[ \mathbf{1}_{\widetilde N_{s_{\indice-1} + \tau_\indice} \neq 0} \mathbb{P}_{(\widetilde N_{s_{\indice-1} + \tau_\indice}, \widetilde S_{s_{\indice-1}+\tau_\indice})}(  \widetilde N_{t_0} \neq 0)\right]\\
&\leq (1-\delta) \mathbb{P}_{(x,s)}(\widetilde N_{s_{\indice-1} + \tau_\indice} \neq 0)
\leq (1-\delta) \mathbb{P}_{(x,s)}(\widetilde N_{s_{\indice-1}} \neq 0)\\
&\leq (1-\delta)^\indice.
\end{align*}
On the other hand, by the Markov property, Equation \eqref{eq:chernoff}, the martingale properties (stopping time theorem on a truncated version of $s_\indice$ and Fatou Lemma) and noting that $\tau=\tau_1$, 
\begin{align*}
\mathbb{P}_{(x,s)}(t \in [s_\indice,s_{\indice+1}) )
&\leq \mathbb{E}_{(x,s)}\left[ \mathbf{1}_{t \geq s_\indice} \mathbb{P}_{(\widetilde N_{s_{\indice}}, \widetilde S_{s_{\indice}})}( \tau \geq t-s_{\indice}  -t_0  \, | \, s_\indice )\right]\\
&\leq \mathbb{E}_{(x,s)}\left[ \mathbf{1}_{t \geq s_\indice} (C_0^{-1}\, V_0(\widetilde N_{s_{\indice}}, \widetilde S_{s_{\indice}})+B) e^{-D (t-s_{\indice}  -t_0)}\right]\\
&\leq C_0^{-1}\, e^{D t_0} e^{-D t} \mathbb{E}_{(x,s)}\left[  V_0(\widetilde N_{s_{\indice}}, \widetilde S_{s_{\indice}}) e^{D s_{\indice} }\right] + B  e^{-D t}  e^{D t_0} \mathbb{E}_{(x,s)}\left[   e^{D s_{\indice}}\right]\\
&\leq  C_0^{-1}\,  e^{D t_0} e^{-D t}  V_0(x,s) + B  e^{-D t}  e^{D t_0} \mathbb{E}_{(x,s)}\left[   e^{D s_{\indice}}\right].
\end{align*}

Moreover by \eqref{eq:momentexpo}, the Markov property, the martingale properties and an induction argument
\begin{align*}
\mathbb{E}_{(x,s)}\left[   e^{D s_{\indice}}\right]
&= e^{D t_0} \mathbb{E}_{(x,s)}\left[   e^{D s_{\indice-1}} \mathbb{E}_{(\widetilde N_{s_{\indice-1}}, \widetilde S_{s_{\indice-1}})}\left[ e^{D \tau} \right]\right]\\
&\leq e^{D t_0} \mathbb{E}_{(x,s)}\left[   e^{D s_{\indice-1}} (C_0^{-1}\, V_0(\widetilde N_{s_{\indice-1}}, \widetilde S_{s_{\indice-1}})+B)\right]\\\
&\leq C_0^{-1}\,  e^{D t_0} V_0(x,s) + B e^{D t_0} \mathbb{E}_{(x,s)}\left[   e^{D s_{\indice-1}}\right]\\
&\leq \left( B e^{D t_0} \right)^\indice \left(1 + \frac{ C_0^{-1}\,  e^{D t_0} V_0(x,s)}{ B e^{D t_0} -1} \right).
\end{align*}

Finally, this gives the existence of a constant $C_1>0$ (which depends on $t_0$ but not on $\theta$) such that 
$$
\mathbb{P}_{(x,s)}(T_0>t)
\leq C_1 e^{-D (1-\theta) t} (C_0^{-1}\,  V_0(x,s) + B) \sum_{\indice \geq 0}  \left((1-\delta)^{\theta} (Be^{Dt_0})^{1-\theta} \right)^\indice.
$$
Choosing $\theta$ sufficiently close to $1$ to guarantee that $(1-\delta)^{\theta} (Be^{Dt_0})^{1-\theta} <1$ ends the proof of \eqref{eq:proba-ext} and then of the statement of the lemma.

\medskip

It remains to prove \eqref{eq:small} to end the proof. Let us introduce
$$
\Psi : (x,s) \mapsto  \frac{2 \sqrt{x}}{\sqrt{\mu(s)+D}}.
$$
We now set $U_t= \Psi(\widetilde{N}_t, \widetilde{S}_t)$ for all $t\geq 0$. This new process hits $0$ at the same time $T_0$ as $(\widetilde{N}_t)_{\geq 0}$  and, using Itô formula, it verifies
\begin{align*}
dU_t 
&=\left[ (\mu(\widetilde{S}_t)-D)\widetilde{N}_t \partial_x \Psi(\widetilde{N}_t, \widetilde{S}_t) + \frac{1}{2} (\mu(\widetilde{S}_t)+D) \widetilde{N}_t \partial_{xx} \Psi(\widetilde{N}_t, \widetilde{S}_t)\right.\\ 
&\left. +  \left(D\,(\Sin-\widetilde S_t)
					-\frac{k}{V\,n}\,m\,\mu(\widetilde S_t)\, \widetilde N_t \right) \partial_{s} \Psi(\widetilde{N}_t, \widetilde{S}_t) \right] dt\\
&+ \sqrt{(\mu(\widetilde{S}_t)+D)\widetilde{N}_t} \partial_x \Psi(\widetilde{N}_t, \widetilde{S}_t) d B_t\\
&= \left[ \frac{1}{2}(\mu(\widetilde{S}_t)-D) U_t - \frac{1}{4 U_t} \right.\\ 
&\left. - \frac{U_t}{2(\mu(\widetilde{S}_t)+D)} \mu'(\widetilde{S}_t) D\,(\Sin-\widetilde S_t)
+ \frac{U_t^3}{8} \mu'(\widetilde{S}_t)   \frac{k}{V\,n}\,m\,\mu(\widetilde S_t)\right] dt\\
&+  d B_t.
\end{align*}
One can then bound the drift term with quantities not depending on the substrate rate and use \cite[Theorem 1.1 chapter VI]{ikeda1981a} to see that $U_t\leq \widetilde{Z}_t$ for every $t\geq 0$, where $(\widetilde{Z}_t)_{t\geq 0}$ is the one-dimensional diffusion solution to
$$
d \tilde{Z}_t = C ( \widetilde{Z}^3_t - \frac{1}{\widetilde{Z}_t}) dt + dB_t
$$
for some constant $C>0$. 
By the Feller's test  for explosions (see \cite[Chapter 5]{KS}), and a monotonicity argument, we deduce that, for all $z_0>0$, there exists $t_0>0$ such that
\begin{equation}
\label{eq:extinctionz}
\inf_{z \leq z_0} \mathbb{P}(\widetilde{T}_0 < t_0\ | \ \widetilde{Z}_0=z) >0,
\end{equation}
where $\widetilde{T}_0=\inf\{t\geq 0  \ | \ \widetilde{Z}_t=0 \}$. More precisely, let $S=\inf\{t\geq 0  \ | \ \widetilde{Z}_t\notin (0, + \infty) \}$ be the exit time from $(0, + \infty)$. The scale function $p$ defined in \cite[Equation (5.42)]{KS} is given by
$$
p: x \mapsto c \int_1^x y^b e^{-a y^4} dy, 
$$
for some $a,b,c>0$, and then $\lim_{x \to 0^+} p(x)$ and $\lim_{x \to + \infty} p(x)$ are clearly finite. Moreover, the function $v$ (defined in \cite[Equation (5.65)]{KS}) verifies
$$
v: x \mapsto \int_1^x p'(y) \int_1^y \frac{2 dz}{ p'(z)} dy= 2\, \int_1^x y^b e^{-a y^4} \int_1^y e^{a z^4} z^{-b} dz dy.
$$
Using standard results of asymptotic analysis, we have 
$$
\int_1^y e^{a z^4} z^{-b} dz  \sim_{ + \infty} e^{a y^4} y^{-b} \times \frac{1}{4a y^3}.
$$
Then as $\int_1^x y^{-3} \dif y <+ \infty$, we have that $\lim_{x \to + \infty} v(x)$ is finite.  
Moreover, for $x \in (0,1)$, we have
$$
0< v(x)= \int_x^1 y^b e^{-a y^4} \int_y^1 e^{a z^4} z^{-b} dz dy \leq e^a \int_x^1 y^b \int_y^1 z^{-b} dz dy \leq   \frac{e^a}{|b-1|} 
$$ then $\lim_{x \to 0^+} v(x)$ is also finite (note that even if the case $b=1$ is not treated in the previous line, it works as well).
As a consequence  by \cite[Proposition 5.32 (i)]{KS}, the stopping time $S$ is finite (and even integrable) and by \cite[Proposition 5.22 (d)]{KS} $\mathbb{P}(S= \widetilde{T}_0 \ | \ \widetilde{Z}_0=z_0)>0$, for every $z_0>0$. 
Consequently, for every $z_0>0$, there exits $t_0>0$ such that $\mathbb{P}( \widetilde{T}_0<t_0 \ | \ \widetilde{Z}_0=z_0)>0$ and then using that for all $z\leq z_0$, 
$$
\mathbb{P}( \widetilde{T}_0 <t_0 \ | \ \widetilde{Z}_0=z) \geq \mathbb{P}( \widetilde{T}_0 <t_0 \ | \ \widetilde{Z}_0=z_0),
$$
we have proved \eqref{eq:extinctionz}. Finally \eqref{eq:small} is a direct consequence of \eqref{eq:extinctionz}.
\end{proof}

\begin{remark}[Quasi-stationary distribution]
Equation \eqref{eq:proba-ext} is a necessary (but not sufficient) condition to ensure existence of a quasi-stationary distribution; see for instance \cite{CMS13}.
\end{remark}

\begin{remark}[Extinction of the Crump-Young model]
It is not difficult to see that \eqref{eq:small} and the Lyapunov property also hold for the Crump-Young model and then \eqref{eq:proba-ext} also holds for this process. In particular this gives a new proof of \cite[Theorem 3.1]{collet2013}. Moreover, in contrast to \cite[Theorem 3.1]{collet2013}, we obtain the speed of extinction \eqref{eq:proba-ext}; furthermore we do not assume any monotonicity on $\mu$. 
\end{remark}

\section*{Aknowledgments}
The authors thank Sylvie Méléard about some discussions on tightness on Hilbert spaces.

This work was partially supported by the Chaire ``Mod\'elisation Math\'ematique et Biodiversit\'e'' of VEOLIA Environment, \'Ecole Polytechnique, Mus\'eum National d'Histoire Naturelle and Fondation X and by the project PIECE (Piecewise Deterministic Markov Processes) of ANR (French national research agency).

\bibliographystyle{apalike}


\end{document}